\documentclass[11pt]{article}       

\usepackage{graphicx}

\usepackage{amssymb}
\usepackage{amsmath}
\usepackage{amsthm}
\usepackage{txfonts}
\usepackage{tabularx}
\usepackage{epsfig}
\usepackage{mathdots}

\usepackage{booktabs}
\usepackage{threeparttable}

\newtheorem{definition}{Definition}
\newtheorem{proposition}{Proposition}
\newtheorem{lemma}{Lemma}
\newtheorem{theorem}{Theorem}
\newtheorem{corollary}{Corollary}

\title{Exponentiation of Graphs}

\author{Toru Hasunuma \\ \\ 
Department of Mathematical Sciences, Tokushima University, \\
2--1 Minamijosanjima, Tokushima 770--8506 Japan}

\date{\today}

\begin{document}

\maketitle

\begin{abstract}
Motivated by very large-scale communication networks,
we newly introduce exponentiation of graphs.
Using the exponential operation on graphs, 
we can construct various graphs of multi-exponential order  
with logarithmic diameter.
We show that every connected exponential graph is maximally connected.
For exponential graphs, we also present a necessary and sufficient condition 
to be super edge-connected
and sufficient conditions to be Hamiltonian
and to have edge-disjoint Hamiltonian cycles
and completely independent spanning trees. 
Applying our results to previously known networks, 
we have maximally connected and super edge-connected 
Hamiltonian graphs of doubly exponential order  
with logarithmic diameter. 
We furthermore define iterated exponential graphs
which may be of not only practical but also theoretical interest. 

\bigskip

\end{abstract}

\section{Introduction}

Throughout this paper, a graph means a simple undirected graph unless stated otherwise. 
Depending on the context,
we may also use the term ``network" with the same meaning as a graph.
Let $G = (V,E)$ be a graph. 
The order of $G$ is the number of vertices in $G$. 
A graph of order 1 (respectively, even order) is called trivial (respectively, even). 
For $v \in V(G)$, $N_G(v)$ denotes 
the set of vertices adjacent to $v$ in $G$.
The degree of a vertex $v$ in $G$ is denoted by ${\rm deg}_G(v)$,
i.e., ${\rm deg}_G(v) = |N_G(v)|$.
Let $\delta(G) = \min_{v \in V(G)}{\rm deg}_G(v)$ and
$\Delta(G) = \max_{v \in V(G)}{\rm deg}_G(v)$.
When $\delta(G) = \Delta(G) = r$, $G$ is called $r$-regular.
For a proper subset $S \subsetneq V(G)$ (respectively, a subset $F \subseteq E(H)$), 
we denote by $G-S$ (respectively, $G-F$) the graph obtained 
from $G$ by deleting every vertex in $S$ (respectively, every edge in $F$). 
For two sets $A$ and $B$, $A \setminus B$ denotes the set 
difference $\{ x \ |\ x \in A, x \not\in B\}$.
For a nonempty subset $S \subseteq V(G)$, the subgraph of $G$ induced by $S$ 
is denoted by $\langle S \rangle_G$, i.e., $\langle S \rangle_G = G-(V(G) \setminus S)$.

The connectivity $\kappa(G)$ of a graph $G$ is the minimum number of vertices whose removal from $G$
results in a disconnected graph or a trivial graph.
A graph $G$ is $k$-connected if $\kappa(G) \geq k$.
The edge-connectivity $\lambda(G)$ of $G$ is similarly defined
by replacing vertex-deletion with edge-deletion.
For a connected graph $G$ and 
$F \subseteq E(G)$, $F$ is an edge-cut of $G$ if $G-F$ is disconnected. 
As a fundamental property, it holds that 
$$\kappa(G) \leq \lambda(G) \leq \delta(G).$$
A graph $G$ is {\it maximally connected} when $\kappa(G) = \delta(G)$. 
Also, $G$ is {\it super edge-connected}
if every minimum edge-cut of $G$ isolates a vertex of $G$,
i.e., for any minimum edge-cut $F$, there exists a vertex $u$ with ${\rm deg}_G(u) = \delta(G)$
such that $F = \{ uv \in E(G)\ |\ v \in N_G(u)\}$.

A walk from a vertex $u = w_0$ to a vertex $v = w_\ell$ in a graph $G$ is a sequence 
$W = (w_0, w_1, \ldots, w_\ell)$ of vertices of $G$ such that every consecutive vertices 
are adjacent in $G$, where $\ell$ is the length of $W$ which is denote by $|W|$, i.e., $\ell = |W|$.
If there is no repeated vertex in a walk, then the walk is called a path.
A trivial path is a path of length 0. 
A cycle is a closed walk, i.e., $w_0 = w_\ell$ such that $(w_0,w_1,\ldots,w_{\ell-1})$
is a path where $\ell \geq 3$.
For a connected graph $G$ and $u,v \in V(G)$, the distance ${\rm dist}_G(u,v)$ 
between $u$ and $v$ in $G$ is the length of a shortest path between them in $G$.
The diameter ${\rm diam}(G)$ of a connected graph $G$ is defined to be 
$\max_{u,v \in V(G)}{\rm dist}_G(u,v)$.
For any positive integers $\Delta$ and $D$, it can be checked that any graph
with maximum degree $\Delta$ and diameter $D$ has order at most
$$1+\Delta+\Delta(\Delta-1)+\cdots+\Delta(\Delta-1)^{D-1}.$$
This upper bound on the order of a graph is called the {\it Moore bound} from which 
it follows that for any bounded-degree graphs of order $n$,
logarithmic diameter $O(\log{n})$ is asymptotically optimal.

Very large-scale communication networks have recently received attention
and a network called DCell has been proposed by Guo et al. \cite{Guo-et-al}. 
The DCell network denoted $D_{k,n}$ is a $(n+k-1)$-regular graph of 
order $t_{k,n}$, where $t_{k,n}$ is recursively defined as follows:
$$\left\{\begin{array}{lcl}
t_{0,n} & = & n, \\
t_{k,n} & = & t_{k-1,n}(t_{k-1,n}+1). 
\end{array} \right.$$
This recurrence follows from the structure of $D_{k,n}$;
$D_{0,n}$ is the complete graph of order $n$ and $D_{k,n}$ consists of
$(t_{k-1,n}+1)$ copies of $D_{k-1,n}$ such that for any two copies of $D_{k-1,n}$,
there is exactly one edge between them. 
Since the definition of $D_{k,n}$ is not so simple, we will omit it here. 
From the recurrence, the following bounds can be obtained \cite{Guo-et-al}.
$$ \left( n+ \frac{1}{2} \right)^{2^{k}} - \frac{1}{2} \leq t_{k,n} \leq (n+1)^{2^{k}}-1.$$
Moreover, it has been shown in \cite{Kli-et-al} that 
$t_{k,n}+\frac{1}{2} < c^{2^{k}} < t_{k,n}+\frac{3}{5}$, i.e, $t_{k,n} = \lfloor c^{2^{k}} \rfloor$, 
where 
$c = \left(n+\frac{1}{2}\right)
\prod_{i = 0}^{\infty}\left(1+\frac{1}{4\left(t_{i,n}+\frac{1}{2}\right)^{2}}\right)^{\frac{1}{2^{i+1}}}$; 
unfortunately, the closed form of the order of $D_{k,n}$ is unknown. 
Thus, $D_{k,n}$ has doubly exponential order. 
It has also been shown in \cite{Guo-et-al} that ${\rm diam}(D_{k,n}) \leq 2^{k+1}-1$.
Thus, $D_{k,n}$ has logarithmic diameter; however, the exact value of ${\rm diam}(D_{k,n})$
remains unknown. 
Logarithmic diameter is a desirable property from 
a point of communication delay. 
On the other hand, from a fault-tolerant point of view, maximally connectedness and 
super edge-connectedness are also known to be desirable properties for communication networks.
It has been proved that the DCell network $D_{k,n}$
is maximally connected \cite{Guo-et-al} 
and super edge-connected for all $k \geq 2$ and $n \geq 2$ \cite{L-W}. 
From these nice properties, the DCell network has been received particular attention
and its various properties have been studied
(e.g., see \cite{Li-et-al,Liu-et-al,Qin-et-al,Wang-et-al,Wang-et-al2,Wang-et-al3}).
In addition, as far as we know, no communication network of 
doubly exponential order with logarithmic diameter has been proposed until now,
except for the DCell network.

Many operations on graphs have been introduced in graph theory so far. 
For example, 
\begin{itemize}
\item unary operations: complementary, line graph, power, subdivided-line graph etc., 
\item binary operations: union, join, Cartesian product, lexicographic product, 
edge sum, etc.
\end{itemize}
These operations are actually treated in \cite{B-C-L,Ha}. 
Motivated by very large-scale communication networks, 
we newly define the {\it exponential operation} on graphs; 
for graphs $G$ and $H$, we denote by $G^H$ the exponential graph
with base $G$ and exponent $H$.
(The precise definition of $G^H$ will be presented in Section 2.) 
Note that ``exponentiation" of graphs is different from ``power" of graphs.
The exponential operation is indeed a binary operation on graphs, while
the $n$-th power $G^n$ of a graph $G$ can be considered as
a unary operation on graphs for any fixed positive integer $n$.
In this paper, we show the following structural results.
\begin{itemize}
\item Every connected exponential graph $G^H$ has logarithmic diameter
if the base $G$ has logarithmic diameter. 
\item Every connected exponential graph $G^H$ is maximally connected.
\item A connected exponential graph $G^H$ is 
super edge-connected if and only if $\delta(G) \geq 2$ or $H$ is not a complete graph,
and $\delta(H) \geq 2$ or $G$ is not a complete graph. 
\end{itemize}
Using the exponential operation, we can construct various maximally connected
and super edge-connected graphs of multi-exponential order with logarithmic diameter.
For example, for the complete graph $K_n$ of order $n$ 
and the $d$-ary de Bruijn graph $B(d,k)$, 
the exponential graph $K_{n}^{B(d,k)}$ is a maximally connected and super edge-connected 
graph of order $n^{d^{k}}d^{k}$ 
with $\Delta\left(K_{n}^{B(d,k)}\right) = n+2d-1$ and 
${\rm diam}\left(K_{n}^{B(d,k)}\right) \leq 2d^{k}+k-1$.  
Compared to the DCell network $D_{k,n}$, $K_{n}^{B(d,k)}$ has advantages that 
the order is explicitly determined and, moreover, 
it can be increased while fixing the maximum degree.

A cycle (respectively, path) containing every vertex of $G$ is called a Hamiltonian cycle
(respectively, Hamiltonian path).
If $G$ has a Hamiltonian cycle, then $G$ is {\it Hamiltonian}.
If for any two vertices in $G$, there exists a Hamiltonian path between them,
then $G$ is {\it Hamiltonian-connected}. 
In this paper, we present the following sufficient conditions to be Hamiltonian and 
to have edge-disjoint Hamiltonian cycles and completely independent spanning trees. 
(The definition of completely independent spanning trees will be presented in Section 5.)
\begin{itemize}
\item If $G$ is Hamiltonian, then $G^{K_2}$ is Hamiltonian.
\item If $G$ is even Hamiltonian and $H$ is Hamiltonian-connected,
then $G^{H}$ is Hamiltonian.
\item If $G$ is even Hamiltonian, then $G^{K_n}$ where $n \geq 4$ 
has two edge-disjoint Hamiltonian cycles and two completely independent spanning trees. 
\end{itemize}
It has been shown in \cite{Ch-et-al,Ka,Se} 
that the second power $G^2$ (respectively, third power $G^3$)
of a 2-connected (respectively, connected) graph $G$ is Hamiltonian-connected.
Combining these results and the above second result, we have the following.
\begin{itemize}
\item If $G$ is even Hamiltonian and $H$ is 2-connected 
(respectively, connected), then $G^{H^2}$ (respectively, $G^{H^3}$) is Hamiltonian.
\end{itemize}
The hypercube $Q_k$ is one of the most well-known communication networks.
The hypercube is, unfortunately, not Hamiltonian-connected, 
while several hypercube-like networks are known to be Hamitonian-connected.
For example, the M\"{o}bius cube $MQ_k$ is one of such networks.
From our results, 
the exponential graph $K_{n}^{MQ_k}$ for even $n$ 
is a maximally connected and super edge-connected $(n+k-1)$-regular 
Hamiltonian graph of order $2^{k}n^{2^{k}}$ with diameter $2^{k+1}$.
Compared to the DCell network $D_{n,k}$, 
$K_{n}^{MQ_k}$ has advantages that
not only the order but also the diameter is explicitly determined.

Since the exponential operation is not associative, 
there are many types of iterated exponential graphs for any given graph $G$.
In particular, we define the {\it exponential cube} $\Omega_k$ and 
{\it hyper-exponential cube} $\Psi_k$ as the following $k$-iterated exponential 
graphs of $K_2$:
$$\Omega_k = \underbrace{\left( \cdots \left(K_{2}^{K_{2}}\right)^{\iddots}\right)^{K_2}}_{k},
\hspace*{10mm}
\Psi_k = \underbrace{K_{2}^{K_{2}^{\iddots^{K_{2}}}}}_{k}. 
$$
These exponential graphs have actually no structure related to ``cube";
the terms are simply derived from an analogous to the hypercube which can be defined as
the Cartesian product of $k$ copies of $K_2$: 
$$Q_k = \underbrace{K_2 \times K_2 \times \cdots \times K_2}_{k}.$$
The exponential cube $\Omega_k$ ($k \geq 3$) is a maximally connected and super edge-connected 
$k$-regular Hamiltonian graph 
of order $2^{2^{k}-1}$ with diameter $3\cdot 2^{k-1}-2$.
From these properties, the exponential cube may be considered as a nice candidate 
for doubly exponential-scale networks. 
The hyper-exponential cube $\Psi_k$ ($k \geq 3$) is also a maximally connected and
super edge-connected $k$-regular graph; however, its order is enormous, i.e.,   
$|V(\Psi_k)| = O\left(\underbrace{2^{2^{\iddots^{2}}}}_{k}\right)$.
Although $\Psi_k$ seems to be impractical unless $k$ is small, it would be of theoretical interest;
in fact, as far as we know,
no such huge networks with logarithmic diameter have been proposed until now.

This paper is organized as follows.
Section 2 presents the definition of the exponential graphs
and their fundamental properties.
Results concerning diameter are given in Section 3.  
In Section 4, we show that every connected exponential graphs are maximally connected
and also present a necessary and sufficient condition to be super edge-connected.  
Sufficient conditions to be Hamiltonian and 
to have edge-disjoint Hamiltonian cycles and completely independent spanning trees
are given in Section 5.
In Section 6, we apply our results to constructions of various networks 
of multi-exponential order with logarithmic diameter
and also compare such networks with the DCell network.
Section 7 finally concludes the paper with several remarks.

\section{Exponential Graphs}

In this section, we first introduce the exponential operation on graphs.

\begin{definition}

Let $G$ and $H$ be graphs such that 
$V(H) = \{w_1,w_2,\ldots,w_q\}$.
The exponential graph $G^H$ with base $G$ and exponent $H$
is defined as follows.
\[ \left\{ \begin{array}{lcl}
V\left(G^H\right) & = & 
\left\{ (u_1,u_2,\ldots,u_q; w_j)\ |\ u_i \in V(G), 1 \leq i \leq q, 
1 \leq j \leq q \right\}, \\[5mm]
E\left(G^H\right) & = & 
\left\{ (u_1,u_2,\ldots,u_q; w_j)(v_1,v_2,\ldots,v_q; w_k)\ \left|\ 
\begin{array}{lll}
j = k, & u_{j}v_{j} \in E(G), & u_i = v_i \mbox{ for any } i \neq j \\[1mm]
& \mbox{ or } & \\[1mm] 
j \neq k, & w_{j}w_{k} \in E(H), & u_i = v_i \mbox{ for all } i
\end{array}
\right. \right\}.
\end{array} \right.
\]
\end{definition}

Note that even if we employ another vertex-labeling $\{w'_1,w'_2,\ldots,w'_q\}$ for $V(H)$,
the structure of $G^H$ remains the same in the sense that
the resultant graphs are isomorphic.
It is essential to fix the correspondence between the vertices of $H$
and the positions of vertices of $G$ in the $(q+1)$-tuples representing the vertices of $G^H$.
We call an edge by the first (respectively, second) condition in the above adjacency rule of $G^H$
a $G$-edge of dimension $j$ or simply a $j$-th $G$-edge
(respectively, an $H$-edge).
For $x \in V(G^H)$, the set of vertices adjacent to $x$
through $G$-edges and $H$-edges are denoted by $N_{G^H}(x; G)$ and 
$N_{G^H}(x; H)$, respectively. 
A path consisting of only $G$-edges (respectively, $H$-edges) in $G^H$
is called a $G$-path (respectively, $H$-path). 
By the definition, when $G$ or $H$ is trivial, i.e, $G \cong K_1$ or $H \cong K_1$,
we have $K_{1}^{H} \cong H$ and $G^{K_1} \cong G$.

\begin{proposition} \label{Prop1}
Let $G$ and $H$ be graphs of orders $p$ and $q$, respectively. 
Then, 
\begin{enumerate}
\item $\left|V\left(G^H\right)\right| = p^{q}q$,
\item $\left|E\left(G^H\right)\right| = p^{q-1}\left(q|E(G)|+p|E(H)|\right)$,
\item $\delta\left(G^H\right) = \delta(G) + \delta(H)$,
\item $\Delta\left(G^H\right) = \Delta(G) + \Delta(H)$.
\end{enumerate}
\end{proposition}

\begin{proof}
The first fact is clear from the definition of $V\left(G^H\right)$.
The number of $j$-th $G$-edges is $p^{q-1}|E(G)|$ for each $1 \leq j \leq q$,
while the number of $H$-edges is $p^{q}|E(H)|$.
Thus, the second fact holds.
Each vertex $x = (u_1,u_2,\ldots,u_q; w_j)$ in $G^H$ is adjacent to
${\rm deg}_{G}(u_j)$ $G$-edges and ${\rm deg}_{H}(w_j)$ $H$-edges, i.e.,
${\rm deg}_{G^H}(x) = {\rm deg}_{G}(u_j) + {\rm deg}_{H}(w_j)$.
Therefore, we have the third and fourth facts. 
\end{proof}

\bigskip

For a $q$-tuple $u = (u_1,u_2,\ldots,u_q)$ where $u_i \in V(G)$ for $1 \leq i \leq q$,
we denote by $H_u$ 
the subgraph of $G^H$ induced by 
the set $\{ (u_1,u_2,\ldots,u_q; w_j)\ |\ w_j \in V(H) \}$, i.e., 
$$H_u = \langle \{ (u_1,u_2,\ldots,u_q; w_j)\ |\ w_j \in V(H) \} \rangle_{G^H}.$$
From the definition, it follows that $H_u \cong H$ and every edge of $H_u$ is an $H$-edge.
For a vertex $x = (u_1,u_2,\ldots,u_q; w_j) \in V(G^H)$, we introduce the following notations: 
\begin{itemize}
\item $\rho(x) =  (u_1,u_2,\ldots,u_q)$,
\item $\rho_i(x) = u_i$ for $1 \leq i \leq q$,
\item $\sigma(x) = j$.
\end{itemize}
Using these notations, we can say that 
for any $x \in V(G^H)$, $x$ is a vertex of $H_{\rho(x)}$ and 
$x$ is incident with ${\rm deg}_H(w_{\sigma(x)})$ $H$-edges
and ${\rm deg}_G(\rho_{\sigma(x)}(x))$ $G$-edges of dimension $\sigma(x)$. 
That is, $|N_{G^H}(x; H)| = {\rm deg}_H(w_{\sigma(x)})$ and 
$|N_{G^H}(x; G)| = {\rm deg}_G(\rho_{\sigma(x)}(x))$. 
For any $G$-edge (respectively, $H$-edge) $xy \in E(G^H)$,
it holds that $\rho(x) \neq \rho(y)$ and $\sigma(x) = \sigma(y)$
(respectively, $\rho(x) = \rho(y)$ and $\sigma(x) \neq \sigma(y)$).

The exponential operation is related to the Cartesian product operation.
The {\it Cartesian product} of $G$ and $H$ denoted $G \times H$
is defined as follows.
\[ \left\{ \begin{array}{lcl}
V(G \times H) & = & \{ (u, w)\ |\ u \in V(G), w \in V(H) \}, \\[3mm]
E(G \times H) & = & \left\{ (u, w)(u',w') \left|\ 
\begin{array}{l}
uu' \in E(G), w = w' \\[1mm]
\mbox{ or } \\[1mm]
u = u', ww' \in E(H) 
\end{array}
\right. \right\}.
\end{array} \right.
\]
Unlike the exponential operation, the Cartesian product is associative.
We then denote by $G^{[n]}$ 
the Cartesian product of $n$ copies of a graph $G$:
$$ G^{[n]} = \underbrace{G \times G \times \cdots \times G}_{n}.$$
By the definition, $G^{[n]}$ can also be defined as follows.
\[ \left\{ \begin{array}{lcl}
V(G^{[n]}) & = & 
\{ (u_1,u_2,\ldots,u_n)\ |\ u_i \in V(G), 1 \leq i \leq n \}, \\
E(G^{[n]}) & = & 
\{ (u_1,u_2,\ldots,u_n)(v_1,v_2,\ldots,v_n)\ |\ 
u_{j}v_{j} \in E(G) \mbox{ for some } j,\ u_i = v_i \mbox{ for any } i \neq j \}.
\end{array} \right.
\]
Thus, $G^{[n]}$ is a graph of order $|V(G)|^n$ such that
$\delta(G^{[n]}) = n \cdot \delta(G)$ and $\Delta(G^{[n]}) = n \cdot \Delta(G)$. 
An edge $(u_1,u_2,\ldots,u_n)(v_1,v_2,\ldots,v_n)$ where 
$u_{j}v_{j} \in E(G)$ in $G^{[n]}$ is called an edge of dimension $j$.

\begin{figure}[t]
\centering
\epsfig{file=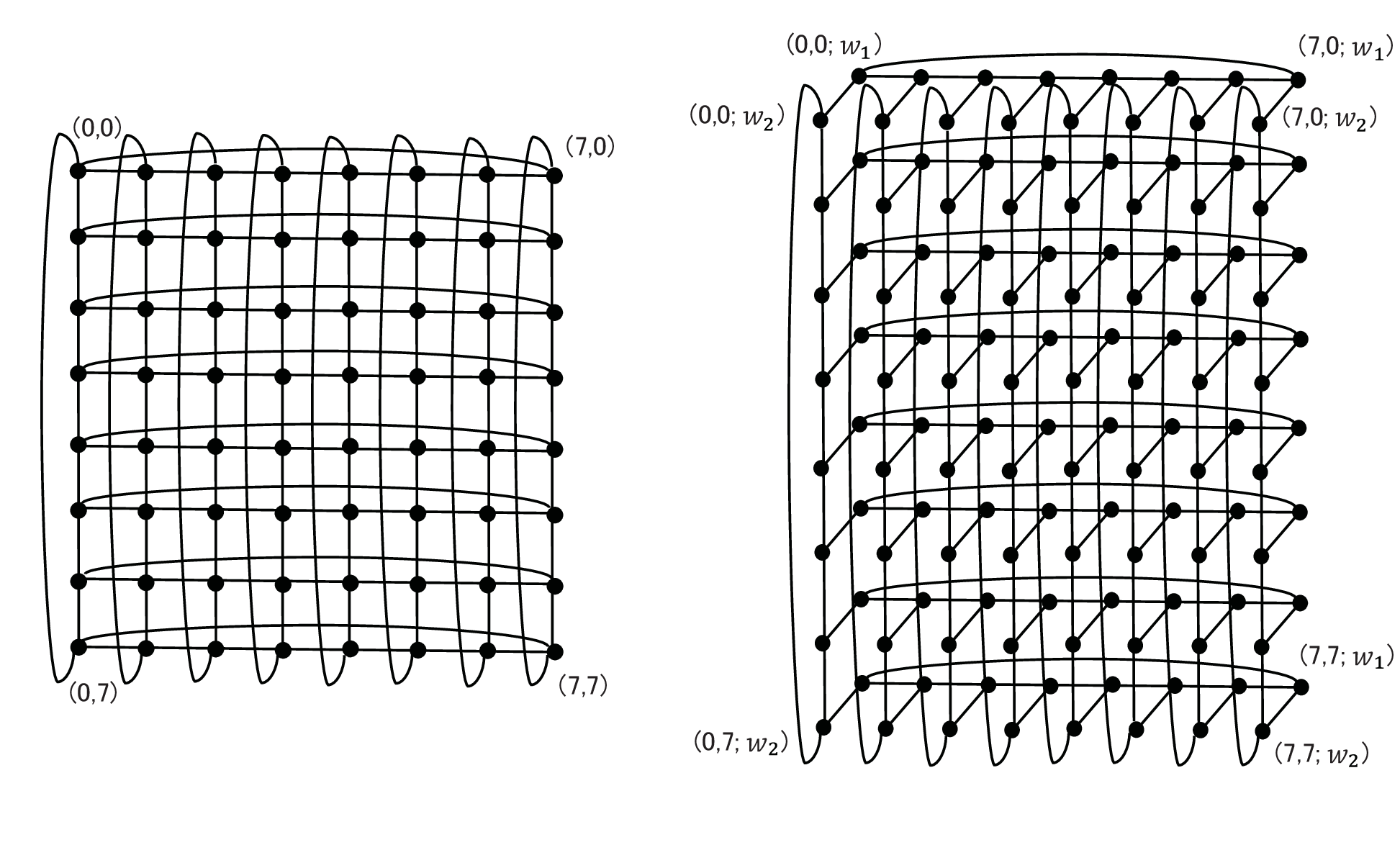,height=70mm}
\caption{The Cartesian product graph $C_8^{[2]}$ and the exponential graph $C_8^{K_2}$.}
\end{figure}

Let $V(H) = \{w_1,w_2,\ldots,w_{q} \}$. 
We can obtain $G^H$ from $G^{[q]}$ by replacing each vertex 
$u$ with a copy $H'$ of $H$ with $V(H') = \{(u;w_1),(u;w_2),\ldots,(u;w_{q}) \}$
in such a way that 
every edge of dimension $j$ incident with $u$ is incident with $(u;w_j)$ as a $G$-edge of 
dimension $j$ for each $1 \leq j \leq q$.
Conversely, given an exponential graph $G^H$,
by contracting the vertices in $V(H_u)$ into the single vertex $u$, 
we have the Cartesian product graph $G^{[q]}$. 

Let $C_n$ denote a cycle of order $n$ such that $V(C_n) = \{0,1,2,\ldots,n-1\}$
and $E(C_n) = \{ 01, 12, 23,\ldots,$ $(n-1)0 \}$.
Fig. 1 shows the Cartesian product graph $C_8^{[2]}$ and the exponential graph $C_8^{K_2}$,
where $V(K_2) = \{w_1,w_2\}$. 
We can see that $C_8^{K_2}$ is obtained from $C_8^{[2]}$ by replacing
each vertex with a copy of $K_2$ so that 
every edge of dimension $j$ incident with $(u_1,u_2)$ is incident with $(u_1,u_2; w_j)$ as 
a $G$-edge of dimension $j$, 
while $C_8^{[2]}$ is obtained from $C_8^{K_2}$ by contracting each copy of $K_2$
into a single vertex.

\section{Diameter}

Let $G$ be a connected graph such that $u,v \in V(G)$ and $S \subseteq V(G)$.
We denote by ${\rm dist}_G(u,v; S)$ the length of a shortest walk from $u$ to $v$
in $G$ which contains every vertex in $S$. 
Note that ${\rm dist}_G(u,v; \emptyset) = {\rm dist}_G(u,v)$.
In particular, 
we call ${\rm dist}_G(u,v; V(G))$ the {\it Hamiltonian-distance} between $u$ and $v$ in $G$.
We then define the {\it Hamiltonian-diameter} ${\rm diam}^\ast(G)$ of a graph $G$ 
as ${\rm diam}^\ast(G) = $ $\max_{u,v \in V(G)}{\rm dist}_G(u,v;V(G))$.

\begin{proposition} \label{Prop2}
For every nontrivial connected graph $H$, 
$$|V(H)| \leq {\rm diam}^\ast(H) \leq 2|V(H)|-2.$$
\end{proposition}
 
\begin{proof}
The lower bound of $|V(H)|$ on ${\rm diam}^\ast(H)$ follows from the fact
that ${\rm dist}_H(u,u; V(H)) \geq |V(H)|$ for $u \in V(H)$. 
For any $u,v \in V(H)$, 
by the following algorithm, we can construct a walk from $u$ to $v$ which contains 
every vertex of $H$, from which the upper bound on ${\rm diam}^\ast(H)$ is obtained.
\begin{enumerate}
\item Construct a shortest path $P$ from $u$ to $v$ in $H$.
\item Contract all the vertices on $P$ into a single vertex $w$.
Let $H'$ be the resultant graph which may have multiple edges. 
\item Construct a spanning tree $T$ of $H'$.
\item Based on $T$, construct a closed walk $W$ from $w$ to $w$ 
of length $2|E(T)| = 2(|V(H)|-|P|-1)$ which contains every vertex of $H'$. 
\item Based on $W$, construct a walk $W'$ from $u$ to $v$ of length 
$|W|+|P| = 2(|V(H)|-1)-|P|$ which contains every vertex of $H$. 
\end{enumerate}
From the above construction, $|W'| \leq 2|V(H)|-2$.
Therefore, we have ${\rm diam}^\ast(H) \leq 2|V(H)|-2$.
\end{proof}

\bigskip

If $H$ is Hamiltonian-connected, then we have 
${\rm diam}^\ast(H) = |V(H)|$, while 
if $H$ is a tree, then ${\rm diam}^\ast(H) = 2|V(H)|-2$.
In this sense, both the upper and lower bounds in Proposition \ref{Prop2} are best possible. 
Note that if $H$ is Hamiltonian, then it holds that ${\rm diam}^\ast(H) \leq |V(H)|-1+{\rm diam}(H)$.

\begin{theorem} \label{DiamTh}
Let $G$ and $H$ be nontrivial connected graphs. 
Then, 
$${\rm diam}\left(G^H\right) = {\rm diam}(G) \cdot |V(H)|+{\rm diam}^\ast(H).$$
\end{theorem}

\begin{proof}
Let $x = (u_1,u_2,\ldots,u_q; w_j)$ and $y = (v_1,v_2,\ldots,v_q; w_k)$
be vertices in $G^H$.
Let 
$$D = \{w_i \in V(H)\ |\ u_i \neq v_i \} =
\{ w_{i_1}, w_{i_2}, \ldots, w_{i_t} \}.$$ 
Also, let 
$W_D$ be a shortest walk from $w_j$ to $w_k$ in $H$ which contains every vertex in $D$.
We may assume, without loss of generality, that 
$w_{i_1}, w_{i_2}, \ldots, w_{i_t}$ appear in $W_D$ in this order.
According to the following algorithm where we assume that $t = 0$ when $D = \emptyset$, 
we can obtain a shortest path from $x$ to $y$
in $G^H$.
\begin{enumerate}
\item Set $\ell := 1$, $i_0 := j$ and $i_{t+1} := k$.
\item Based on $W_D$, construct a path $P_{H,\ell}$ from 
$(u_1,u_2, \ldots, u_q; w_{i_{\ell-1}})$ to $(u_1,u_2, \ldots, u_q; w_{i_\ell})$. 
\item If $\ell \leq t$, then 
\begin{enumerate}
\item[3.1.] construct a path $P_{G,\ell}$ from
$(u_1,\ldots, u_{i_\ell}, \ldots, u_q; w_{i_{\ell}})$ to $(u_1,\ldots, v_{i_\ell}, \ldots, u_q; w_{i_\ell})$
based on a shortest path from $u_{i_\ell}$ to $v_{i_\ell}$ in $G$, 
\item[3.2.] set $\ell := \ell+1$ and go to Step 2.
\end{enumerate}
\item By concatenating the paths $P_{H,1}$, $P_{G,1}$, $P_{H,2}, \ldots, P_{G,t}$ and $P_{H, t+1}$ 
in this order, construct a path $P^\ast$ from $x$ to $y$ in $G^H$.
\end{enumerate}
Note that $P_{H,1}$ or $P_{H, t+1}$ may be trivial.
From the above construction, 
$$|P^\ast| = |W_D| + \sum_{\ell = 1}^{t}|P_{G,i}| = 
{\rm dist}_{H}(w_j, w_k; D)+ \sum_{\ell = 1}^{t} {\rm dist}_G(u_{i_\ell},v_{i_\ell}).$$
Thus, we have ${\rm diam}\left(G^H\right) \leq {\rm diam}^\ast(H)+{\rm diam}(G) \cdot |V(H)|$
and in fact we can select two vertices in $G^H$ whose distance 
equals ${\rm diam}^\ast(H)+{\rm diam}(G) \cdot |V(H)|$.
\end{proof}

\bigskip

From Proposition \ref{Prop2} and Theorem \ref{DiamTh},
we have the following corollary.

\begin{corollary} \label{DiamCor1}
Let $G$ and $H$ be nontrivial connected graphs. 
Then, 
$$({\rm diam}(G)+1) \cdot |V(H)| \leq {\rm diam}\left(G^H\right) \leq ({\rm diam}(G)+2) \cdot |V(H)|-2.$$
\end{corollary}

Assuming that $H$ is a tree, Hamiltonian or Hamiltonian-connected, we also have the following.
In particular, when $H$ is Hamiltonian, based on a fixed Hamiltonian cycle in $H$,
a simple routing algorithm for vertices of $G^H$ is obtained from the proof of Theorem \ref{DiamTh}.  

\begin{corollary} \label{DiamCor2}
Let $G$ and $H$ be nontrivial connected graphs.
\begin{enumerate}
\item If $H$ is a tree, then 
${\rm diam}\left(G^H\right) = ({\rm diam}(G)+2) \cdot |V(H)|-2$. 
\item If $H$ is Hamiltonian, then
${\rm diam}\left(G^H\right) \leq \left({\rm diam}(G)+1\right) \cdot |V(H)| + {\rm diam}(H)-1$. 
\item If $H$ is Hamiltonian-connected, then 
${\rm diam}\left(G^H\right) = ({\rm diam}(G)+1) \cdot |V(H)|$. 
\end{enumerate}
\end{corollary}

From Corollary \ref{DiamCor1}, the following theorem is obtained.

\begin{theorem}
Let $G$ and $H$ be nontrivial connected graphs.
If $G$ has logarithmic diameter, then 
$G^H$ has logarithmic diameter.
\end{theorem}

\begin{proof}
Since $|V(G^H)| = |V(G)|^{|V(H)|} \cdot |V(H)|$, 
$$\log{\left|V\left(G^H \right)\right|} = |V(H)| \cdot \log{|V(G)|}+\log{|V(H)|}.$$
From Corollary \ref{DiamCor1}, ${\rm diam}\left(G^H\right) \leq ({\rm diam}(G)+2) \cdot |V(H)|-2$.
Thus, if $G$ has logarithmic diameter, i.e., ${\rm diam}(G) = O(\log{|V(G)|})$, then 
$${\rm diam}\left(G^H\right) = O\left(\log{\left|V\left(G^H \right)\right|}  \right),$$
i.e., $G^H$ has logarithmic diameter. 
\end{proof}

\section{Connectivity}

As mentioned in Section 2, the structure of an exponential graph $G^H$ is related to
that of the Cartesian product graph $G^{[|V(H)|]}$.
The connectivity of a Cartesian product graph $G \times H$ has been determined by 
\v{S}pacapan as follows.

\begin{theorem}  {\rm (\v{S}pacapan \cite{Sp})} \label{SpT}
For any nontrivial graphs $G$ and $H$, 
$$\kappa(G \times H) = \min\{ \kappa(G) \cdot |V(H)|,\ \kappa(H) \cdot |V(G)|,\ 
\delta(G)+\delta(H) \}.$$
\end{theorem}

Suppose that $G$ is nontrivial and connected. 
From Theorem \ref{SpT},
$\kappa(G^{[2]}) = \min\{\kappa(G) \cdot |V(G)|, 2\delta(G)\}$. 
If $G \cong K_2$ or $\kappa(G) \geq 2$, 
then the right-hand-side in this equation equals $2\delta(G)$.
If $G \not\cong K_2$ and $\kappa(G) = 1$, then 
by applying the following famous theorem by Dirac, 
we have $\delta(G) < \frac{|V(G)|}{2}$ which implies that 
the right-hand-side also equals $2\delta(G)$.
Note that if $\kappa(G) = 1$, then $G$ is not Hamiltonian. 
Thus, it holds that $\kappa(G^{[2]}) = 2\delta(G)$ for any nontrivial connected graph $G$.

\begin{theorem} {\rm (Dirac \cite{D})}
Every graph $G$ of order at least 3 with $\delta(G) \geq \frac{|V(G)|}{2}$ is Hamiltonian.
\end{theorem}

Again, from Theorem \ref{SpT} under the condition $n \geq 3$
with the induction hypothesis $\kappa(G^{[n-1]}) = (n-1)\delta(G)$, we have 
$$\begin{array}{lcl}
\kappa(G^{[n]}) = \kappa(G \times G^{[n-1]}) 
& = & \min\{ \kappa(G) \cdot |V(G^{[n-1]})|,\ \kappa(G^{[n-1]}) \cdot |V(G)|,\ 
\delta(G)+\delta(G^{[n-1]}) \} \\ 
& = & \min\{ \kappa(G) \cdot |V(G)|^{n-1},\ (n-1)\delta(G) \cdot |V(G)|,\ 
\delta(G)+(n-1)\delta(G)\} \\
& = & \min\{ \kappa(G) \cdot |V(G)|^{n-1},\ n \cdot \delta(G)\}. \\
\end{array}$$
Since 
$$\kappa(G) \cdot |V(G)|^{n-1} \geq (\delta(G)+1)^{n-1} \geq \delta(G)^{n-1}+(n-1)\delta(G)^{n-2}+1 > n \cdot \delta(G),$$
the following corollary holds.

\begin{corollary} \label{ConCarP}
For any nontrivial connected graph $G$ and any 
$n \geq 2$, 
$G^{[n]}$ is maximally connected, i.e., 
$$\kappa\left(G^{[n]}\right) = n \cdot \delta(G).$$
\end{corollary}

Based on Corollary \ref{ConCarP},
we show the following.

\begin{theorem} \label{MaxConTh}
For any nontrivial connected graphs $G$ and $H$,
the exponential graph $G^H$ is maximally connected, i.e.,
$$\kappa\left(G^H\right) = \delta(G)+\delta(H).$$
\end{theorem}

\begin{proof}
Suppose that $G$ and $H$ are nontrivial and connected and $|V(H)| = q$. 
Let $S \subset V(G^H)$ such that $|S| = \delta(G^H)-1 = \delta(G)+\delta(H)-1$.
Since $\kappa(G^H) \leq \delta(G^H)$,
it is sufficient to show that $G^H-S$ is connected.

Define $S'$ as follows:
$$S' = \{ u \in V(G^{[q]})\ |\ S \cap V(H_u) \neq \emptyset \}.$$
Clearly, $|S'| \leq |S|$. 
By Corollary \ref{ConCarP}, $\kappa(G^{[q]}) = q \cdot \delta(G)$.
Since $$q \cdot \delta(G) \geq \delta(G)+(q-1)\delta(G) \geq \delta(G)+\delta(H),$$
$G^{[q]}-S'$ is connected.
Thus $G^H-\cup_{u \in S'}V(H_u)$ is connected. 
Let $x \in \cup_{u \in S'}V(H_u) \setminus S$.
If $$N_{G^H}(x; G) \setminus \cup_{u \in S'}V(H_u) \neq \emptyset,$$
then $x$ is adjacent to a vertex in $G^H-\cup_{u \in S'}V(H_u)$.
Suppose that $N_{G^H}(x; G) \setminus \cup_{u \in S'}V(H_u) = \emptyset$,
i.e., $$N_{G^H}(x; G) \subseteq \cup_{u \in S'}V(H_u).$$
For any two distinct vertices $y_1,y_2 \in N_{G^H}(x; G)$, $\rho(y_1) \neq \rho(y_2)$, i.e.,
$H_{\rho(y_1)} \neq H_{\rho(y_2)}$.
Thus, 
$$|S \cap (\cup_{y \in N_{G^H}(x; G)}V(H_{\rho(y)}))| \geq |N_{G^H}(x;G)| \geq \delta(G).$$
Since $|S| = \delta(G)+\delta(H)-1$ and $|N_{G^H}(x;H)| \geq \delta(H)$, 
$$N_{G^H}(x; H) \setminus S \neq \emptyset.$$
Let $z \in N_{G^H}(x; H) \setminus S$.
If $$N_{G^H}(z; G) \setminus \cup_{u \in S'}V(H_u) \neq \emptyset,$$
then $z$ is adjacent to a vertex in $G^H-\cup_{u \in S'}V(H_u)$
and $x$ is connected to the vertex through $z$.
Suppose that $N_{G^H}(z; G) \subseteq \cup_{u \in S'}V(H_u)$.
Since $|S \cap (\cup_{y \in N_{G^H}(z; G)}V(H_{\rho(y)})| \geq \delta(G)$
and $N_{G^H}(z;G) \cap N_{G^H}(x;G) = \emptyset$, 
it holds that $\delta(H)-1 \geq \delta(G)$ and 
$$|N_{G^H}(x; H) \cap S| \leq \delta(H)-1-\delta(G),$$
i.e., 
$$|N_{G^H}(x; H) \setminus S| \geq \delta(G)+1 \geq 2.$$
Thus, we can select a vertex $z' \in N_{G^H}(x;H) \setminus (S \cup \{z\})$.
If $N_{G^H}(z'; G) \setminus \cup_{u \in S'}V(H_u) \neq \emptyset$,
then $x$ is connected to a vertex in $G^H-\cup_{u \in S'}V(H_u)$.
Otherwise, we similarly have $$|N_{G^H}(x;H) \setminus S| \geq 2\delta(G)+1 \geq 3$$
and there exists a vertex $z'' \in N_{G^H}(x;H) \setminus (S \cup \{z,z'\})$.
Iteratively applying a similar discussion, we finally find a vertex 
$z^\ast \in N_{G^H}(x;H) \setminus S$ such that $x$ is connected to 
a vertex in $G^H-\cup_{u \in S'}V(H_u)$ through $z^\ast$. 
Therefore, $G^H-S$ is connected.
\end{proof}

\bigskip

From Theorem \ref{MaxConTh}, we have $\lambda(G^H) = \delta(G^H)$ for nontrivial connected graphs
$G$ and $H$. 
We next characterize super edge-connected exponential graphs.

\begin{theorem} \label{SupEdTh}
Let $G$ and $H$ be nontrivial connected graphs.  
The exponential graph $G^H$ is super edge-connected
if and only if 
$\delta(G) \geq 2$ or $H$ is not a complete graph,
and $\delta(H) \geq 2$ or $G$ is not a complete graph.
\end{theorem}

\begin{proof}
Let $G$ and $H$ be connected graphs with $V(G) = \{1,2,\ldots,p\}$
and $V(H) = \{w_1,w_2,\ldots,w_q\}$ where $p \geq 2$ and $q \geq 2$.
Suppose that $\delta(G) \geq 2$ or $H \not\cong K_q$,
and $\delta(H) \geq 2$ or $G \not\cong K_p$. 
Let $F \subset E(G^H)$ with $|F| = \delta(G^H)$ such that 
$F$ does not isolate any vertex in $G^H$. 
We then show that $G^H-F$ is connected.

Let $F_G = \{ xy \in F\ |\ \rho(x) \neq \rho(y) \}$ and 
$F_H = \{xy \in F\ |\ \rho(x) = \rho(y) \}$.
Define $S_G$ and $S_H$ as follows:  
$$\left\{ \begin{array}{lll}
S_G & = & \left\{ \left. \rho(x) \in V(G^{[q]})\ \right|\ xy \in F_G, 
\rho_{\sigma(x)}(x) < \rho_{\sigma(y)}(y) \right\}, \\
S_H & = & \left\{ \left. \rho(x) \in V(G^{[q]})\ \right|\ xy \in F_H \right\}.
\end{array} \right. $$
Note that $|S_G \cup S_H| \leq |F|$ and $G^H - \cup_{u \in S_G \cup S_H}V(H_u) \subset G^H-F$.
From the condition that $\delta(G) \geq 2$ or $H \not\cong K_q$, 
$$q \cdot \delta(G) \geq \delta(G)+(q-1)\delta(G) > \delta(G)+\delta(H) = \delta(G^H).$$
Since $\kappa(G^{[q]}) = q \cdot \delta(G)$ by Corollary \ref{ConCarP}, 
$G^{[q]}-(S_G \cup S_H)$ is connected.
Therefore, $G^H - \cup_{u \in S_G \cup S_H}V(H_u)$ is connected. 

Let $x \in V(G^H)$ such that $\rho(x) \in S_G \cup S_H$. 
Now suppose that for any edge $xy \in E(G^H) \setminus F$ incident with $x$, 
\begin{enumerate}
\item if $xy$ is a $G$-edge, then $\rho(y) \in S_G \cup S_H$, 
\item if $xy$ is an $H$-edge, then for any $G$-edge $yz$ incident with $y$, 
$yz \in F$ or $\rho(z) \in S_G \cup S_H$. 
\end{enumerate}
Note that the vertex $x$ is connected to a vertex of
$G^H - \cup_{u \in S_G \cup S_H}V(H_u)$ in $G^H-F$, unless these conditions hold. 
From the conditions, it follows that ${\rm deg}_{G^H}(x) = \delta(G^H)$
and $S_G \cup S_H \subseteq N_{G^{[q]}}(\rho(x)) \cup \{\rho(x)\}$.
Moreover, it holds that for any $G$-edge $xy \not\in F$ incident with $x$,
the number of edges in $F$ incident with a vertex in $V(H_{\rho(y)}) \setminus \{y\}$ 
is at most one, since for any vertex $y' \in V(H_{\rho(y)}) \setminus \{y\}$,
there is no $G$-edge $y'z$ incident with $y'$ such that
$\rho(z) \in N_{G^{[q]}}(\rho(x)) \cup \{ \rho(x) \}$.

\bigskip

\noindent Case 1: $\delta(G) \geq 2$.

The existence of an $H$-edge $xy \not\in F$ implies that 
$|F| > \delta(G^H)$.
Thus, every $H$-edge incident with $x$ is in $F$ 
and there exists a $G$-edge incident with $x$ which is not in $F$.
Except for the case that $\delta(H) = 1$ and for any $G$-edge $xy \not\in F$, 
the $H$-edge $yy' \in E(H_{\rho(y)})$ is in $F$, 
there exists a path from $x$ to a vertex of $G^H - \cup_{u \in S_G \cup S_H}V(H_u)$
in $G^H-F$. 
Consider the exceptional case. 
Since $G \not\cong K_p$ and ${\rm deg}_{G}(\rho_{\sigma(x)}(x)) = \delta(G)$,
there exists a $G$-path $P_1 = (x,y,z_2)$ or $P_2 = (x,y,z_1,z_2)$ 
in $G^H-F$ such that $\rho(z_2) \not\in N_{G^{[q]}}(\rho(x)) \cup \{ \rho(x) \}$.
Thus, $x$ is connected to a vertex of $G^H - \cup_{u \in S_G \cup S_H}V(H_u)$ in $G^H-F$. 

\bigskip

\noindent Case 2: $\delta(G) = 1$.

Suppose that there exists an $H$-edge $xy \not\in F$.
Since $H \not\cong K_q$ and ${\rm deg}_{H}(w_{\sigma(x)}) = \delta(H)$, 
there exists a $H$-path $Q_1 = (x,y,y_2)$ or $Q_2 = (x,y,y_1,y_2)$ in $G^H-F$
such that $xy_2 \not\in E(G^H)$ and 
for any $G$-edge $y_2y'_2$ incident with $y_2$, $\rho(y'_2) \not\in S_G \cup S_H$.
Thus, $x$ is connected to a vertex of $G^H - \cup_{u \in S_G \cup S_H}V(H_u)$ in $G^H-F$. 
Suppose that every $H$-edge incident with $x$ is in $F$.
Since $H \not\cong K_2$ when $\delta(H) = 1$, 
similarly to Case 1, it can be shown that $x$ is connected to a vertex 
of $G^H - \cup_{u \in S_G \cup S_H}V(H_u)$ in $G^H-F$.

\bigskip

From the above discussion, 
any vertex $x$ in $H_{\rho(x)}$ where $\rho(x) \in S_G \cup S_H$
is connected to a vertex of 
$G^H - \cup_{u \in S_G \cup S_H}V(H_u)$ in $G-F$.
Therefore, $G-F$ is connected.
Hence, $G$ is super edge-connected.

In what follows, we show that $G^H$ is not super edge-connected
when $\delta(G) = 1$ and $H \cong K_q$,
or $\delta(H) = 1$ and $G \cong K_p$. 
First, consider the case that 
$\delta(G) = 1$ and $H \cong K_q$. 
Let $a \in V(G)$ such that ${\rm deg}_G(a) = 1$ and $ab \in E(G)$.
Define $F_1$ as follows:
$$ \begin{array}{ll}
F_1 = \{ & ((a,a,\ldots,a); w_1)((b,a,\ldots,a);w_1), \\
&  ((a,a,\ldots,a); w_2)((a,b,\ldots,a);w_2), \\
& \hspace*{23mm} \vdots \\
& ((a,a,\ldots,a); w_q)((a,a,\ldots,b);w_q) \ \}.
\end{array}$$
Then, $G^H-F_1$ is disconnected; in fact, 
$F_1$ separates $G^H$ into $H_{(a,a,\ldots,a)}$
and $G^{H}-H_{(a,a,\ldots,a)}$. 
Since $|F_1| = q = \delta(G)+\delta(H) = \delta(G^H)$, $F_1$ is a minimum edge-cut which 
does not isolates a vertex.

Next, consider the case that $G \cong K_p$ and $\delta(H) = 1$. 
Without loss of generality, we may assume that ${\rm deg}_H(w_1) = 1$ and $w_1w_2 \in E(H)$.
Let $(u_2,\ldots,u_p) \in V(G^{[p-1]})$. 
Define $F_2$ as follows:
$$ \begin{array}{ll}
F_2  = \{ & ((1,u_2,\ldots,u_p); w_1)((1,u_2,\ldots,u_p);w_2),\\
&  ((2,u_2,\ldots,u_p); w_1)((2,u_2,\ldots,u_p);w_2), \\
& \hspace*{26mm} \vdots \\
& ((p,u_2,\ldots,u_p); w_1)((p,u_2,\ldots,u_p);w_2) \ \}.
\end{array}$$
Then, $G^H$ is separated by $F_2$ into 
$G' = \langle \{ (i,u_2,\ldots,u_p); w_1)\ |\ 1 \leq i \leq p \} \rangle_{G^H} \cong G$ and 
$G^{H}-G'$, i.e., $G^H-F_2$ is disconnected. 
Since $|F_2| = p = \delta(G^H)$, $F_2$ is 
a minimum edge-cut which does not isolate a vertex. 
\end{proof}

\bigskip

From Theorem \ref{SupEdTh},
we have the following corollary.

\begin{corollary}
Every connected exponential graph $G^H$ with $\delta(G) \geq 2$ and $\delta(H) \geq 2$
is super edge-connected.
\end{corollary}

\section{Hamiltonicity}

In this section, we present sufficient conditions for exponential graphs to be Hamiltonian
and to have two edge-disjoint Hamiltonian cycles and also two completely independent
spanning trees. 

We first show the following theorem.

\begin{theorem} \label{Ham-Th1}
If a graph $G$ is Hamiltonian, then the exponential graph $G^{K_2}$ is Hamiltonian.
\end{theorem}

\begin{proof}
Let $V(K_2) = \{w_1,w_2\}$.
Let $G$ be a Hamiltonian graph with $V(G) = \{0,1,\ldots,p-1 \}$ such that 
$(0,1,\ldots,p-1,0)$ is a Hamiltonian cycle in $G$.
In what follows, we suppose that any vertex in $G$ is expressed modulo $p$.
Based on the Hamiltonian cycle, 
for each $a \in V(G)$, we define paths $P_{i}(a)$
consisting of only $G$-edges of dimension $1$ in $G^{K_2}$ as follows:  
$$P_{i}(a) = ((i,a; w_1), (i-1,a; w_1),\ldots,(i-p+1,a; w_1)),\ 0 \leq i < p.$$
For each $a \in V(G)$, we similarly define paths $Q_{i}(a)$ consisting of only $G$-edges of 
dimension $2$ in $G^{K_2}$ as follows: 
$$Q_{i}(a) = ((a,i ; w_2), (a, i+1; w_2),\ldots,(a, i+p-1; w_2)),\ 0 \leq i < p.$$
Note that for any $0 \leq i < p$, $0 \leq j < p$ and $a,b \in V(G)$,
$$ \left\{ \begin{array}{llll}
V(P_{i}(a)) \cap V(Q_{j}(b)) & = & \emptyset, & \\
V(P_{i}(a)) \cap V(P_{j}(b)) & = & \emptyset & \mbox{if $a \neq b$}, \\
V(Q_{i}(a)) \cap V(Q_{j}(b)) & = & \emptyset & \mbox{if $a \neq b$}.
\end{array} \right. $$
Combining paths $P_{i}(a)$, $Q_{j}(a)$ and $H$-edges 
$(a,a'; w_1)(a,a'; w_2)$, we construct a cycle
in $G^{K_2}$ as follows, where ``$\xrightarrow[P_{i}(a)]{}$" and ``$\xrightarrow[Q_{i}(a)]{}$"
indicate moving through the paths $P_{i}(a)$ and $Q_{i}(a)$, respectively. 
$$
\begin{array}{lcccr}
(0,0; w_1) & \xrightarrow[P_{0}(0)]{} & (1,0; w_1), 
(1,0; w_2) & \xrightarrow[Q_{0}(1)]{} & (1,p-1;w_2),\\
(1,p-1; w_1) & \xrightarrow[P_{1}(p-1)]{} & (2,p-1;w_1),
(2,p-1; w_2) & \xrightarrow[Q_{p-1}(2)]{} & (2,p-2;w_2),\\
(2,p-2;w_1) & \xrightarrow[P_{2}(p-2)]{} & (3,p-2; w_1),
(3,p-2; w_2) & \xrightarrow[Q_{p-2}(3)]{} & (3,p-3;w_2),\\
(3,p-3;w_1) & \xrightarrow[P_{3}(p-3)]{} & (4,p-3; w_1),
(4,p-3; w_2) & \xrightarrow[Q_{p-3}(4)]{} & (4,p-4;w_2),\\
  &  \vdots &   &  \vdots &   \\
(p-2,2;w_1) & \xrightarrow[P_{p-2}(2)]{} & (p-1,2; w_1),
(p-1,2; w_2) & \xrightarrow[Q_{2}(p-1)]{} & (p-1,1;w_2),\\
(p-1,1;w_1) & \xrightarrow[P_{p-1}(1)]{} & (0,1; w_1),
(0,1; w_2) & \xrightarrow[Q_{1}(0)]{} & (0,0;w_2),\\ 
(0,0;w_1). & & & & \\
\end{array}
$$
Since 
$$(\cup_{0 \leq i < p}V(P_{i}(p-i)) \cup (\cup_{0 \leq i < p}V(Q_{i}(p+1-i))  
= V\left(G^{K_2} \right),$$ 
the above construction induces a Hamiltonian cycle in $G^{K_2}$.
\end{proof}

\bigskip

\begin{figure}[t]
\centering
\epsfig{file=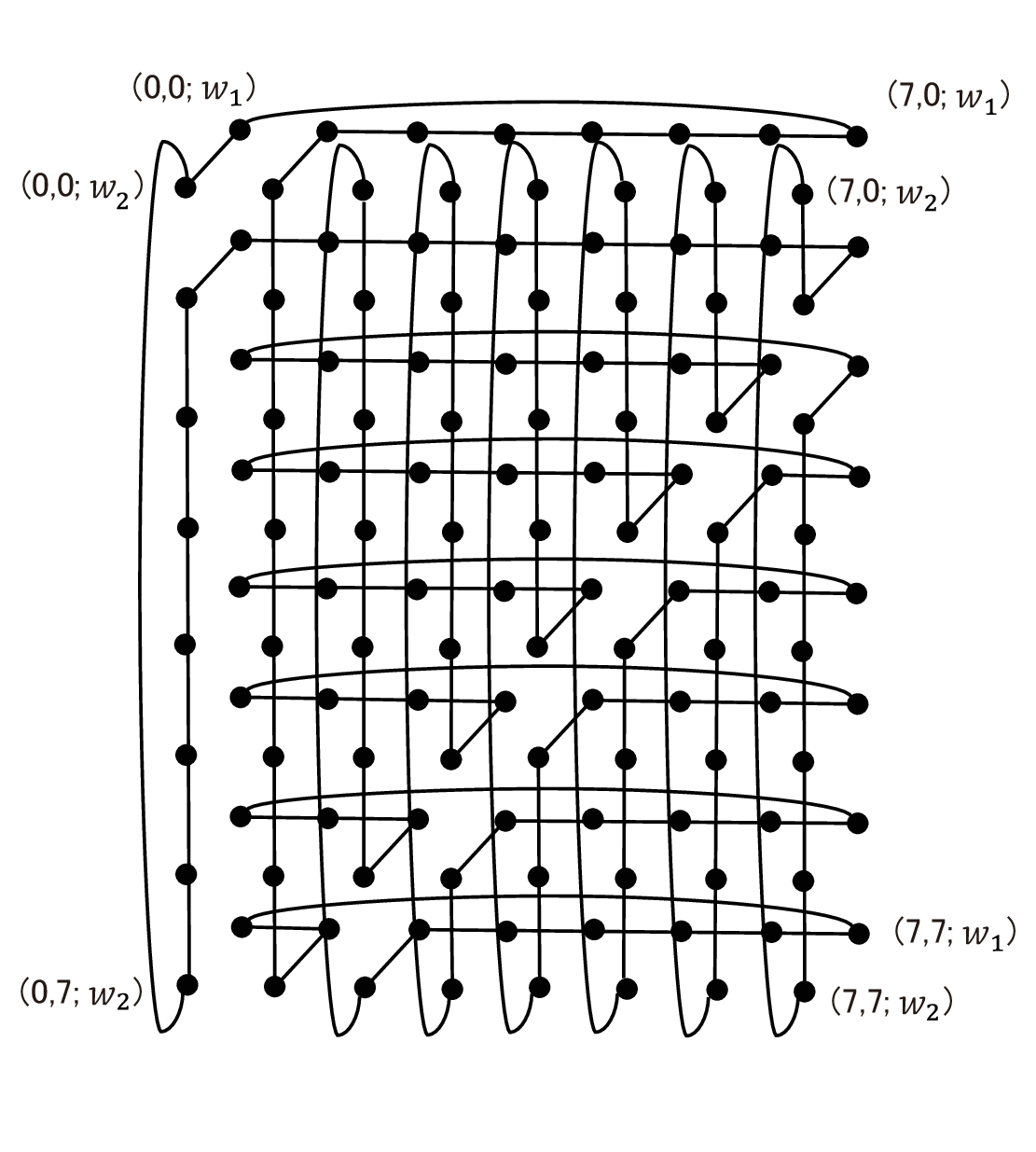,height=70mm}
\caption{The Hamiltonian cycle 
by the construction in the proof of Theorem \ref{Ham-Th1} for the case $p = 8$.}
\end{figure}

Fig. 2 illustrates the Hamiltonian cycle by 
the construction in the proof of Theorem \ref{Ham-Th1} for the case that $p = 8$.

\begin{lemma} \label{Ham-Lem1}
If a graph $G$ is even Hamiltonian,
then for any $n \geq 2$, 
the Cartesian product graph $G^{[n]}$ has a Hamiltonian cycle in which no two consecutive edges have
the same dimension.
\end{lemma}

\begin{proof}
Let $G$ be an even Hamiltonian graph with $V(G) = \{0,1,\ldots,p-1 \}$
such that
$(0,1,\ldots,p-1,0)$ is a Hamiltonian cycle in $G$. 
We first construct a desired Hamiltonian cycle when $n = 2$.
Define a closed walk $C$ as follows, where ``$\xrightarrow[1]{}$" 
and ``$\xrightarrow[2]{}$" indicate a move by increasing the first and second elements
in a 2-tuple corresponding to a vertex of $G^{[2]}$, 
respectively, and  ``$\xrightarrow[2-]{}$" indicates a move by decreasing the second element.

$$
\begin{array}{clclclclcl} 
& (0,0) & \xrightarrow[1]{} & (1,0) & \xrightarrow[2]{} & (1,1)
& \xrightarrow[1]{} \cdots \xrightarrow[2]{} & (p-1,p-1) 
& \xrightarrow[1]{} & (0,p-1) \\
\xrightarrow[2-]{} 
& (0,p-2) & \xrightarrow[1]{} & (1,p-2) & \xrightarrow[2]{} & (1,p-1)
& \xrightarrow[1]{} \cdots \xrightarrow[2]{} & (p-1,p-3) 
&\xrightarrow[1]{} & (0,p-3) \\
\xrightarrow[2-]{} 
& (0,p-4) & \xrightarrow[1]{} & (1,p-4) & \xrightarrow[2]{} & (1,p-3)
& \xrightarrow[1]{} \cdots \xrightarrow[2]{} & (p-1,p-5) 
&\xrightarrow[1]{} & (0,p-5) \\
\xrightarrow[2-]{} 
& (0,p-6) & \xrightarrow[1]{} & (1,p-6) & \xrightarrow[2]{} & (1,p-5)
& \xrightarrow[1]{} \cdots \xrightarrow[2]{} & (p-1,p-7) 
& \xrightarrow[1]{} & (0,p-7) \\
& \ \ \ \ \ \vdots & & \ \ \ \ \ \vdots & & \ \ \ \ \ \vdots & 
& \ \ \ \ \ \vdots & & 
\ \ \ \ \ \vdots \\ 
\xrightarrow[2-]{} 
& (0,2) & \xrightarrow[1]{} & (1,2) & \xrightarrow[2]{} & (1,3)
& \xrightarrow[1]{} \cdots \xrightarrow[2]{} & (p-1,1) 
& \xrightarrow[1]{} & (0,1) \\
\xrightarrow[2-]{} 
& (0,0). & & & & & & & & \\ 
\end{array}
$$
It can be checked that $C$ is a desired Hamiltonian cycle in $G^{[2]}$.
Fig. 3 shows the Hamiltonian cycle $C$ for the case that $p = 8$.

Based on the Hamiltonian cycle $C$, we construct a Hamiltonian cycle $D$ in $G^{[3]}$ as follows,
where ``$\xrightarrow[C]{}$" (respectively, ``$\xrightarrow[C^{R}]{}$") 
indicates moving through edges of dimensions 1 or 2 according to $C$ in order (respectively, reverse order) 
and 
``$\xrightarrow[3]{}$" indicate a move by increasing the third element in a 3-tuple
corresponding to a vertex of $G^{[3]}$. 

$$
\begin{array}{clclclcl}
&  (0,0,0) & \xrightarrow[C]{} & (0,1,0) & \xrightarrow[3]{} & (0,1,1) & 
\xrightarrow[C^{R}]{} & (0,0,1) \\
\xrightarrow[3]{} 
&  (0,0,2) & \xrightarrow[C]{} & (0,1,2) & \xrightarrow[3]{} & (0,1,3) & 
\xrightarrow[C^{R}]{} & (0,0,3) \\
\xrightarrow[3]{} 
&  (0,0,4) & \xrightarrow[C]{} & (0,1,4) & \xrightarrow[3]{} & (0,1,5) & 
\xrightarrow[C^{R}]{} & (0,0,5) \\
& \ \ \ \ \ \vdots & & \ \ \ \ \ \vdots & & \ \ \ \ \ \vdots & & 
\ \ \ \ \ \vdots \\ 
\xrightarrow[3]{} 
&  (0,0,p-2) & \xrightarrow[C]{} & (0,1,p-2) 
& \xrightarrow[3]{} & (0,1,p-1) & 
\xrightarrow[C^{R}]{} & (0,0,p-1) \\
\xrightarrow[3]{} 
&  (0,0,0). & & & & & & \\
\end{array}
$$ 
Note that no two consecutive edges have the same dimension in $D$, i.e.,
$D$ is a desired Hamiltonian cycle in $G^{[3]}$.
Based on $D$,
we can similarly construct a desired Hamiltonian cycle in $G^{[4]}$ as follows. 

$$
\begin{array}{clclclcl}
&  (0,0,0,0) & \xrightarrow[D]{} & (0,0,p-1,0) & \xrightarrow[4]{} & (0,0,p-1,1) & 
\xrightarrow[D^{R}]{} & (0,0,0,1) \\
\xrightarrow[4]{} 
&  (0,0,0,2) & \xrightarrow[D]{} & (0,0,p-1,2) & \xrightarrow[4]{} & (0,0,p-1,3) & 
\xrightarrow[D^{R}]{} & (0,0,0,3) \\
\xrightarrow[4]{} 
&  (0,0,0,4) & \xrightarrow[D]{} & (0,0,p-1,4) & \xrightarrow[4]{} & (0,0,p-1,5) & 
\xrightarrow[D^{R}]{} & (0,0,0,5) \\
& \ \ \ \ \ \vdots & & \ \ \ \ \ \vdots & & \ \ \ \ \ \vdots & & 
\ \ \ \ \ \vdots \\ 
\xrightarrow[4]{} 
&  (0,0,0,p-2) & \xrightarrow[D]{} & (0,0,p-1,p-2) 
& \xrightarrow[4]{} & (0,0,p-1,p-1) & 
\xrightarrow[D^{R}]{} & (0,0,0,p-1) \\
\xrightarrow[4]{} 
&  (0,0,0,0). & & & & & & \\
\end{array}
$$ 

\begin{figure}[t]
\centering
\epsfig{file=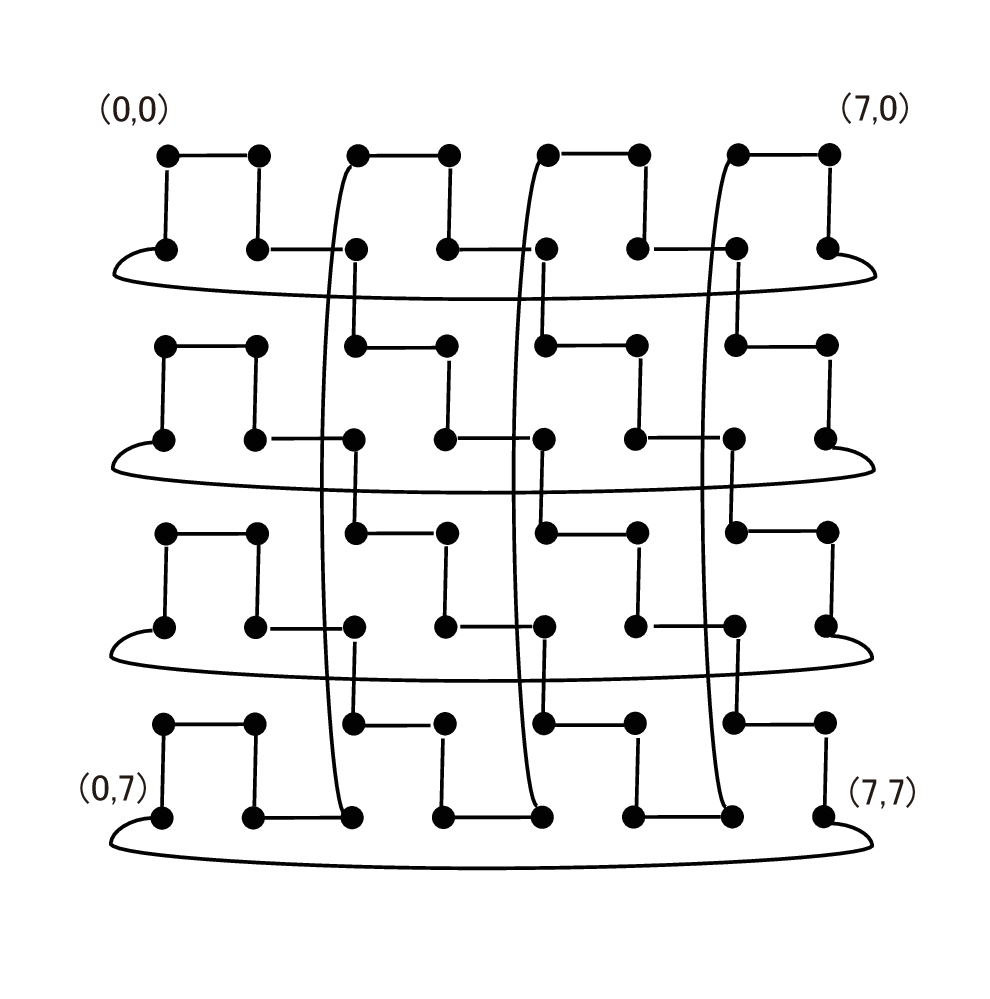,height=60mm}
\caption{The Hamiltonian cycle $C$ by the construction in the proof of Lemma \ref{Ham-Lem1} for
the case $p = 8$.}
\end{figure}

Iteratively applying a similar construction, we can obtain
a desirable Hamiltonian cycle in $G^{[n]}$ for any $n \geq 5$.
\end{proof}

\bigskip

Using Lemma \ref{Ham-Lem1},
we show the following theorem. 

\begin{theorem} \label{Ham-Th2}
If a graph $G$ is even Hamiltonian and a graph $H$ is Hamiltonian-connected,
then the exponential graph $G^H$ is Hamiltonian.
\end{theorem}

\begin{proof}
Suppose that $G$ is an even Hamiltonian graph and 
$H$ is a Hamiltonian-connected graph of order $q$.
By Lemma \ref{Ham-Lem1}, $G^{[q]}$ has a Hamiltonian cycle $C^\ast$ 
in which no two consecutive edges have the same dimension.
Let $S$ be the set of $G$-edges in $E(G^H)$ corresponding to the edges in $C^\ast$.
For each $H_u$ where $u \in V(G^{[q]})$ in $G^H$, 
there are exactly two $G$-edges in $S$ such that
they are incident with distinct vertices $x$ and $y$ in $H_u$.
Since $H$ is Hamiltonian-connected, there is a Hamiltonian path $P_u$ from $x$ to $y$ in $H_u$.
Combining the $G$-edges in $S$ and $P_u$ for each $H_u$ where $u \in V(G^{[q]})$, 
we have a Hamiltonian cycle in $G^H$.
\end{proof}

\bigskip

In the proof of Lemma \ref{Ham-Lem1},
we actually construct a Hamiltonian cycle in which no two consecutive edges have
the same dimension such that every edge of dimension $i \geq 2$ is adjacent to only 
an edge of dimension 1.
Thus, the assumption that $H$ is Hamiltonian-connected in Theorem \ref{Ham-Th2} 
can be weakened to
the condition that there exists a vertex $w$ in $H$ such that for any vertex $v \neq w$,
there exists a Hamiltonian path from $w$ to $v$.

From Theorem \ref{Ham-Th2}, we have the following corollary which 
extends Theorem \ref{Ham-Th1} for even graphs.

\begin{corollary} \label{Ham-Cor}
If a graph $G$ is even Hamiltonian, then
the exponential graph $G^{K_n}$ is even Hamiltonian for any $n \geq 2$. 
\end{corollary}

For any positive interger $n$, the $n$-th {\it power} of a graph $G$ denoted $G^n$ is defined as follows:
$$\left\{
\begin{array}{ll}
V(G^n) & = V(G), \\
E(G^n) & = \{ uv \ |\ {\rm dist}_G(u,v) \leq n \}.
\end{array} \right.$$
In particular, $G^2$ and $G^3$ are called the {\it square} and {\it cube} of $G$, respectively.
For the square and cube of a graph, 
the following fundamental results concerning Hamiltonian-connectedness have been shown.

\begin{theorem} {\rm (Chartrand et. al \cite{Ch-et-al})}
For any 2-connected graph $G$, the square $G^2$ is Hamiltonian-connected.
\end{theorem}

\begin{theorem} {\rm (Karaganis \cite{Ka}, Sekanina \cite{Se})}
For any connected graph $G$, the cube $G^3$ is Hamiltonian-connected.
\end{theorem}

From these results and Theorem \ref{Ham-Th2},
we have the following corollary. 

\begin{corollary}
If a graph $G$ is even Hamiltonian and a graph $H$ is 2-connected (respectively, connected), 
then the exponential graph $G^{H^2}$ (respectively, $G^{H^3}$) is Hamiltonian.
\end{corollary}

When the order of $G$ is at least 4, Lemma \ref{Ham-Lem1} can be significantly extended
as follows.

\begin{lemma} \label{Ham-Lem2}
If $G$ is an even Hamiltonian graph, 
then for any $n \geq 2$, 
the Cartesian product graph $G^{[n]}$ has two edge-disjoint Hamiltonian cycles 
in which no two consecutive edges have the same dimension.
\end{lemma}

\begin{proof}
Let $G$ be an even Hamiltonian graph with a Hamiltonian cycle 
$(0,1,\ldots,p-1,0)$ where $p \geq 4$. 
Define a closed walk $C'$ as follows.
$$
\begin{array}{clclclclcl} 
& (0,0) & \xrightarrow[1-]{} & (p-1,0) & \xrightarrow[2-]{} & (p-1,p-1)
& \xrightarrow[1-]{} \cdots \xrightarrow[2-]{} & (1,1) 
& \xrightarrow[1-]{} & (0,1) \\
\xrightarrow[2]{} 
& (0,2) & \xrightarrow[1-]{} & (p-1,2) & \xrightarrow[2-]{} & (p-1,1)
& \xrightarrow[1-]{} \cdots \xrightarrow[2-]{} & (1,3) 
&\xrightarrow[1-]{} & (0,3) \\
\xrightarrow[2]{} 
& (0,4) & \xrightarrow[1-]{} & (p-1,4) & \xrightarrow[2-]{} & (p-1,3)
& \xrightarrow[1-]{} \cdots \xrightarrow[2-]{} & (1,5) 
&\xrightarrow[1-]{} & (0,5) \\
\xrightarrow[2]{} 
& (0,6) & \xrightarrow[1-]{} & (p-1,6) & \xrightarrow[2-]{} & (p-1,5)
& \xrightarrow[1-]{} \cdots \xrightarrow[2-]{} & (1,7) 
& \xrightarrow[1-]{} & (0,7) \\
& \ \ \ \ \ \vdots & & \ \ \ \ \ \vdots & & \ \ \ \ \ \vdots & 
& \ \ \ \ \ \vdots & & \ \ \ \ \ \vdots \\ 
\xrightarrow[2]{} 
& (0,p-2) & \xrightarrow[1-]{} & (p-1,p-2) & \xrightarrow[2-]{} & (p-1,p-3)
& \xrightarrow[1-]{} \cdots \xrightarrow[2-]{} & (1,p-1) 
& \xrightarrow[1-]{} & (0,p-1) \\
\xrightarrow[2]{} 
& (0,0). & & & & & & & &  \\ 
\end{array}
$$
It can be checked that $C'$ is a Hamiltonian cycle in which no two consecutive edges
have the same dimension.
Note that the moving rules in $C'$ can be obtained from those in the Hamiltonian cycle $C$
in the proof of Lemma \ref{Ham-Lem1} by exchanging `` $\xrightarrow[i]{}$" and ``$\xrightarrow[i-]{}$"
for each $i \in \{1,2\}$.
Suppose that $(i,j)(i,j+1) \in E(C') \cap E(C)$.
Then, from the moving rules of $C'$ and $C$, it follows that $C' = C$.
Similarly, if $(i,j)(i+1,j) \in E(C') \cap E(C)$, then we have $C' = C$.
Since $(0,0)(1,0) \in E(C) \setminus E(C')$, it is concluded that $E(C) \cap E(C') = \emptyset$.
Therefore, $G^{[2]}$ has two edge-disjoint Hamiltonian cycles 
in which no two consecutive edges have the same dimension.
Fig. 4 illustrates the Hamiltonian cycle $C'$ for the case that $p = 8$.
(In particular, when $G \cong C_p$, $G^{[2]}$ is decomposed into two Hamiltonian cycles,
i.e., $E(G) = E(C') \cup E(C)$.)

Using the Hamiltonian path between $(p-1,p-1)$ and $(p-1,0)$ in $C'$, 
we construct a Hamiltonian cycle $D'$ in $G^{[3]}$ as follows.
Note that for $D$ in the proof of Lemma \ref{Ham-Lem1},
we use the Hamiltonian path between $(0,0)$ and $(0,1)$.
Since $\{(p-1,p-1),(p-1,0)\} \cap \{(0,0),(0,1)\} = \emptyset$,
$D'$ and $D$ are edge-disjoint.

$$
\begin{array}{clclclcl}
&  (p-1,p-1,0) & \xrightarrow[C']{} & (p-1,0,0) & \xrightarrow[3]{} & (p-1,0,1) & 
\xrightarrow[(C')^{R}]{} & (p-1,p-1,1) \\
\xrightarrow[3]{} 
&  (p-1,p-1,2) & \xrightarrow[C']{} & (p-1,0,2) & \xrightarrow[3]{} & (p-1,0,3) & 
\xrightarrow[(C')^{R}]{} & (p-1,p-1,3) \\
\xrightarrow[3]{} 
&  (p-1,p-1,4) & \xrightarrow[C']{} & (p-1,0,4) & \xrightarrow[3]{} & (p-1,0,5) & 
\xrightarrow[(C')^{R}]{} & (p-1,p-1,5) \\
& \ \ \ \ \ \vdots & & \ \ \ \ \ \vdots & & \ \ \ \ \ \vdots & & 
\ \ \ \ \ \vdots \\ 
\xrightarrow[3]{} 
&  (p-1,p-1,p-2) & \xrightarrow[C']{} & (p-1,0,p-2) 
& \xrightarrow[3]{} & (p-1,0,p-1) & 
\xrightarrow[(C')^{R}]{} & (p-1,p-1,p-1) \\
\xrightarrow[3]{} 
&  (p-1,p-1,0). & & & & & & \\
\end{array}
$$ 

Iteratively applying a similar construction, we can obtain
a desirable Hamiltonian cycle in $G^{[n]}$ which is edge-disjoint with the one
obtained in the proof of Lemma \ref{Ham-Lem1} for any $n \geq 4$.
Therefore, we have the desired result.
\end{proof}

\begin{figure}[t]
\centering
\epsfig{file=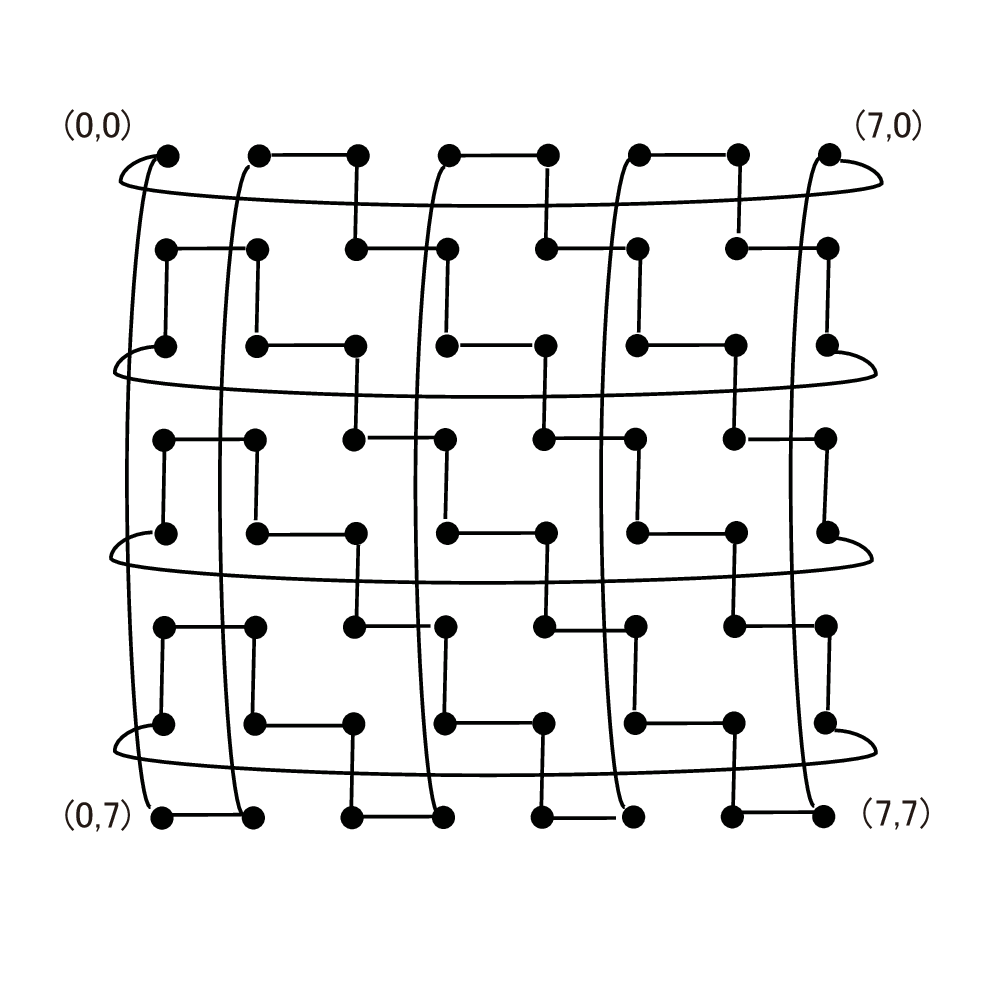,height=60mm}
\caption{The Hamiltonian cycle $C'$ by the construction in the proof of Lemma \ref{Ham-Lem2} for
the case $p = 8$.}
\end{figure}

\bigskip

Based on Lemma \ref{Ham-Lem2}, we have the following result which 
strengthens Corollary \ref{Ham-Cor} when $n \geq 4$. 

\begin{theorem} \label{Ham-Th2}
If $G$ is an even Hamiltonian graph, 
then the exponential graph $G^{K_n}$ where $n \geq 4$ has two edge-disjoint Hamiltonian cycles.
\end{theorem}

\begin{proof}
Suppose that $G$ is an even Hamiltonian graph of order $p \geq 4$ and $n \geq 4$.
By Lemma \ref{Ham-Lem2}, $G^{[n]}$ has two edge-disjoint Hamiltonian cycle $C_1$ and $C_2$ 
in which no two consecutive edges have the same dimension.
Let $S$ be the set of $G$-edges in $E(G^{K_n})$ corresponding to the edges in 
$E(C_1) \cup E(C_2)$.
For each $(K_n)_u$ where $u \in V(G^{[n]})$ in $G^{K_n}$, 
there are exactly four $G$-edges in $S$ which 
are incident with distinct four vertices $x_1,y_1,x_2,y_2$ in $(K_n)_u$
such that $x_1,y_1$ (respectively, $x_2,y_2$) 
are incident with the $G$-edges corresponding to the edges incident with $u$ 
in $C_1$ (respectively, $C_2$).
For $V(K_n) = \{0,1,\ldots,n-1\}$, define $P_1$ and $P_2$ as follows.
$$
\left\{ \begin{array}{ll}
P_1 & = (0,1,n-1,2,n-3,3,\ldots,\lceil \frac{n}{2} \rceil), \\
P_2 & = (1,2,0,3,n-1,\ldots,\lceil \frac{n}{2} \rceil+1).
\end{array} \right.$$
Then, $P_1$ and $P_2$ are edge-disjoint Hamiltonian paths connecting distinct pairs of vertices
$\{0, \lceil \frac{n}{2} \rceil \}$ and $\{1, \lceil \frac{n}{2} \rceil+1\}$. 
From the symmetry of $K_n$,
it follows that there are a Hamiltonian path $P_{u,1}$ between $x_1$ and $y_1$
and a Hamiltonian path $P_{u,2}$ between $x_2$ and $y_2$ in $(K_n)_u$ such that they are edge-disjoint.
Combining the $G$-edges corresponding to $E(C_i)$ in $S$ 
and $P_{u,i}$ for each $(K_n)_u$ where $u \in V(G^{[n]})$ for $i = 1,2$, 
we have two edge-disjoint Hamiltonian cycles in $G^{K_n}$.
\end{proof}

{\em Completely independent spanning trees} in a graph $G$
are spanning trees of $G$ such that for any two distinct vertices $u, v$ of $G$,
the paths between $u$ and $v$ in the spanning trees  
mutually have no common edge and no common vertex except for $u$ and $v$. 
Completely independent spanning trees can be applied to fault-tolerant broadcasting 
\cite{H01} and protection routings \cite{PCC}.
Motivated by these applications, completely independent spanning trees have been 
studied for the DCell network \cite{Qin-et-al} and 
other various graphs (e.g., see \cite{CCW,CPHYC,Ha2,HM}).
Completely independent spanning trees are characterized as follows.

\begin{theorem} {\rm (Hasunuma \cite{H01})} \label{CIST}
Spanning trees $T_1,T_2,\ldots,T_k$ in a graph $G$ are completely independent
if and only if
$T_1,T_2,\ldots,T_k$ are edge-disjoint and
for any vertex $v \in V(G)$, 
there exists at most one spanning tree $T_i$ such that ${\rm deg}_{T_i}(v) > 1$.
\end{theorem}

Using Lemma \ref{Ham-Lem2} and Theorem \ref{CIST},
we show the following theorem.

\begin{theorem} \label{Ham-Th3}
If $G$ is an even Hamiltonian graph, 
then the exponential graph $G^{K_n}$ where $n \geq 4$ has 
two completely independent spanning trees. 
\end{theorem}

\begin{proof}
We employ the notations 
$C_1, C_2, S$ and $x_1,y_1,x_2,y_2$ for each $(K_n)_u$ where $u \in V(G^{[n]})$ 
in the proof of Theorem \ref{Ham-Th2} with the same meanings.
For $V((K_n)_u) = \{x_1,y_1,x_2,y_2,z_1,\ldots,z_{n-4}\}$,
define spanning trees $T_{u,1}$ and $T_{u,2}$ of $(K_n)_u$ as follows.
$$
\left\{ \begin{array}{ll}
E(T_{u,1}) & = \{ (x_1z_i \ | \ 1 \leq i \leq n-4\} 
\cup \{x_1y_1, x_1x_2, y_1y_2\},\\
E(T_{u,2}) & = \{ (x_2z_i\ | \ 1 \leq i \leq n-4\} 
\cup \{x_2y_2, x_2y_1, y_2x_1\}.
\end{array} \right.$$
By the definition, $T_{u,1}$ and $T_{u,2}$ are edge-disjoint such that
for any $v \in \{x_1,y_1\}$, ${\rm deg}_{T_{u,1}}(v) > 1$ and ${\rm deg}_{T_{u,2}}(v) = 1$,
for any $v \in \{x_2,y_2\}$, ${\rm deg}_{T_{u,1}}(v) = 1$ and ${\rm deg}_{T_{u,2}}(v) > 1$,
and for any $v \in \{z_1,\ldots,z_{n-4}\}$, ${\rm deg}_{T_{u,1}}(v) = 1$ and ${\rm deg}_{T_{u,2}}(v) = 1$;
thus, $T_{u,1}$ and $T_{u,2}$ are completely independent spanning trees in $(K_n)_u$
for each $u \in V(G^{[n]})$. 

For each $i = 1,2$,
let $T_i$ be a spanning tree obtained by 
combining the $G$-edges corresponding to the edges of $C_i$ in $S$ 
and the edges of $T_{u,i}$ for each $(K_n)_u$ where $u \in V(G^{[n]})$ 
and deleting one $G$-edge corresponding to an edge of $C_i$ in $S$.
By the definition, $T_1$ and $T_2$ are edge-disjoint.
Moreover, for each $i = 1,2$,
any $G$-edge of $T_i$ is incident with $x_i$ or $y_i$ for each $(K_n)_u$.
Therefore, for any vertex $v \in V(G^{K_n})$, 
it does not hold that ${\rm deg}_{T_1}(v) > 1$ and ${\rm deg}_{T_2}(v) > 1$.
Hence, $T_1$ and $T_2$ are completely independent spanning trees in $G^{K_n}$.
\end{proof}

\section{Applications to Communication Networks}

In this section, we apply our results in the previous sections 
to constructions of various maximally connected and super edge-connected 
networks of multi-exponential order with logarithmic diameter.
We also compare such networks with the DCell network.

\subsection{De Bruijn and Kautz Networks}

The de Bruijn and Kautz networks are well-known as communication networks
with logarithmic diameter and their various properties have been investigated \cite{B-P}.
Although these networks are originally defined as directed graphs; in fact, they can be defined
as iterated line digraphs using the line digraph operation. 
we here define them as undirected graphs since we only treat undirected graphs in this paper.

The de Bruijn network $B(d,k)$ has two parameters $d \geq 2$ and $k \geq 1$ such that
each vertex is a $k$-tuple consisting of nonnegative intergers at most $d-1$.
Two vertices are adjacent in $B(d,k)$ if and only if one is obtained from another by one shifting.
\[ \left\{ \begin{array}{l}
V(B(d,k)) = \{ (v_1,v_2,\ldots,v_k)\ |\ 0 \leq v_i < d, 1 \leq i \leq k \}, \\
E(B(d,k)) = \{ (v_1,v_2,\ldots,v_k)(v_2,\ldots,v_k,\alpha)\ |\ 0 \leq v_i < d, 1 \leq i \leq k,
0 \leq \alpha < d \}.
\end{array} \right.
\]
For any given two vertices $u$ and $v$ of $B(d,k)$, 
$v$ can be obtained from $u$ by at most $k$ shiftings,
which implies that ${\rm diam}(B(d,k)) \leq k$.
In fact, it holds that ${\rm dist}_{B(d,k)}((i,i,\ldots,i),(j,j,\ldots,j)) = k$ for any $0 \leq i < j < d$.
Thus, ${\rm diam}(B(d,k))$ is determined to be $k$.
Since $B(d,k)$ has order $d^k$, $B(d,k)$ has logarithmic diameter.
From the definition, it follows that $\delta(B(d,k)) = 2d-2$ and $\Delta(B(d,k)) = 2d$.
Note that vertices of form $(i,i,\ldots,i)$ and $(i,j,i,j,\ldots)$ where $i \neq j$ have 
degree $2d-2$ and $2d-1$, respectively, and any other vertex has degree $2d$.
In particular, we can increase the orders of de Bruijn networks while fixing the maximum degree,
i.e., de Bruijn networks are bounded-degree networks.
The de Bruijn network is also known to be Hamiltonian.

The Kautz network $K(d,k)$ can be similarly defined as the de Bruijn network
under the condition that no two consecutive numbers are the same in the $k$-tuples corresponding
to the vertices.
\[ \left\{ \begin{array}{l}
V(K(d,k)) = \{ (v_1,v_2,\ldots,v_k)\ |\ 0 \leq v_i \leq d, 1 \leq i \leq k, v_j \neq v_{j+1}, 1 \leq j < k \}, \\
E(K(d,k)) = \{ (v_1,v_2,\ldots,v_k)(v_2,\ldots,v_k,\alpha)\ |\ 0 \leq v_i \leq d, 1 \leq i \leq k,
0 \leq \alpha \leq d, \alpha \neq v_k \}.
\end{array} \right.
\]
By the definition, $K(d,k)$ has order $d^{k}+d^{k-1}$, diameter $k$, minimum degree $2d-1$
and maximum degree $2d$.
Thus, the Kautz network has more vertices than the de Bruijn network 
with the same diameter and maximum degree.
The Kautz network is also known to be Hamiltonian.

Employing the de Bruijn or Kautz network as an exponent for base graphs with logarithmic diameter,
we can obtain maximally connected and super edge-connected 
networks of doubly exponential order with logarithmic diameter.
Fundamental properties of such exponential graphs with base $K_n$ 
are summarized in Table 1.
Note that upper bounds on their diameters follow from the second statement in
Corollary \ref{DiamCor2}.

\begin{table}[t]
\caption{Fundamental properties of $K_{n}^{B(d,k)}$ and $K_{n}^{K(d,k)}$.}
\begin{threeparttable}
\begin{tabular*}{\textwidth}{@{\extracolsep\fill}lcc}
\toprule
& $K_{n}^{B(d,k)}$ & $K_{n}^{K(d,k)}$ \\
\midrule
Order & $n^{d^{k}}d^{k}$ & $n^{d^{k}+d^{k-1}}(d^{k}+d^{k-1})$  \\  
Minimum degree &  $n+2d-3$ & $n+2d-2$  \\
Maximum degree &  $n+2d-1$ & $n+2d-1$  \\
Diameter &  $\leq 2d^{k}+k-1$ & $\leq 2(d^{k}+d^{k-1})+k-1$  \\ 
Connectivity &  $n+2d-3$ & $n+2d-2$  \\
\bottomrule
\end{tabular*}
\end{threeparttable}
\end{table}

Comparison of the DCell network $D_{k,n}$, $K_{n}^{B(2,k)}$ and $K_{n}^{K(2,k)}$ is given in Table 2.
These exponential graphs have an advantage that their orders are
explicitly determined.
Another advantage is that the order can be increased while fixing the maximum degree,
which is inherited from the de Bruijn and Kautz networks.
On the other hand, their disadvantage is that Hamiltonicity is unknown. 

\begin{table}[b]
\caption{Comparison of $D_{k,n}$, $K_{n}^{B(2,k)}$ and $K_{n}^{K(2,k)}$.}
\begin{threeparttable}
\begin{tabular*}{\textwidth}{@{\extracolsep\fill}lccc}
\toprule
& $D_{k,n}$ & $K_{n}^{B(2,k)}$ & $K_{n}^{K(2,k)}$   \\
\midrule
Order & $\leq (n+1)^{2^{k}}-1$ & $2^{k}n^{2^{k}}$ & $3 \cdot 2^{k-1}n^{3 \cdot 2^{k-1}}$ \\
Minimum degree & $n+k-1$ & $n+1$ & $n+2$ \\
Maximum degree & $n+k-1$ & $n+3$ & $n+3$  \\
Diameter
& $\leq 2^{k+1}-1$ & $\leq 2^{k+1}+k-1$ & $\leq 3\cdot 2^{k}+k-1$  \\
Connectivity & $n+k-1$ & $n+1$ & $n+2$  \\
Hamiltonicity & Yes & $-$ & $-$  \\ 
\bottomrule
\end{tabular*}
\end{threeparttable}
\end{table}

By employing the exponential graphs $K_{n}^{B(2,k)}$ and $K_{n}^{K(2,k)}$ as exponents for base $K_n$,
we can furthermore obtain multi-exponential scale networks with logarithmic diameter.
Fundamental properties of such networks are summarized in Table 3.
Note that these exponential graphs also preserve the bounded-degree property.

\begin{table}[t]
\caption{Fundametal properties of 
$K_{n}^{\left( K_{n}^{B(2,k)} \right)}$ and $K_{n}^{\left( K_{n}^{K(2,k)}\right)}$.}
\begin{threeparttable}
\begin{tabular*}{\textwidth}{@{\extracolsep\fill}lcc}
\toprule
& $K_{n}^{\left( K_{n}^{B(2,k)}\right)}$  
& $K_{n}^{\left( K_{n}^{K(2,k)}\right)}$ \\ 
\midrule
Order 
& $2^{k}n^{2^{k}(n^{2^{k}}+1)}$ 
& $3\cdot 2^{k-1}n^{3 \cdot 2^{k-1}(n^{3\cdot 2^{k-1}}+1)}$
\\ 
Minimum degree 
& $2n$
& $2n+1$ \\ 
Maximum degree 
& $2n+2$
& $2n+2$ \\ 
Diameter 
& $\leq 3 \cdot 2^{k}n^{2^{k}}-2$
& $\leq 9 \cdot 2^{k-1}n^{3 \cdot 2^{k-1}}-2$ \\ 
Connectivity 
& $2n$
& $2n+1$ \\
\bottomrule
\end{tabular*}
\end{threeparttable}
\end{table}

\subsection{Hypercube-like Networks}

The hypercube $Q_n$ is one of the most well-known communication networks; 
$Q_n$ has order $2^n$ and diameter $n$, i.e., logarithmic diameter, 
and $Q_n$ is $n$-regular, maximally connected, super edge-connected ($n \neq 2$)
\cite{Lu-et-al}
and Hamiltonian.
Unfortunately, $Q_n$ is not Hamiltonian-connected, which follows from a fact that
$Q_n$ is bipartite. 
Thus, we cannot apply Theorem \ref{Ham-Th2} to any exponential graph with exponent $Q_n$.
However, there are several hypercube-like networks which are 
Hamiltonian-connected \cite{Fan,Fan-et-al,Ma-et-al}, 
namely, the twisted cube $TQ_n$ \cite{H-K-S}, the crossed cube $CQ_n$ \cite{C-L} 
and the M\"{o}bius cube $MQ_n$ \cite{Ef}.
These hypercube-like networks have order $2^{n}$ and 
diameter $\lceil \frac{n+1}{2} \rceil$ 
which is about a half of that of the hypercube,
and they are $n$-regular, maximally connected 
and super edge-connected (except for
some cases) \cite{Chen-et-al}. 
Since the definitions of these networks are not so simple compared to 
the hypercube, we will omit them here.

Comparison of the DCell network, $K_{n}^{TQ_k}$, $K_{n}^{CQ_k}$, $K_{n}^{MQ_k}$, 
$Q_{n-1}^{TQ_k}$, $Q_{n-1}^{CQ_k}$ and $Q_{n-1}^{MQ_k}$ for even $n$ is given in Table 4. 
Note that all the networks have the same degree and connectivity. 
These exponential graphs have advantages that not only the order but also the diameter
is explicitly determined. 
Since $K_n$ and $Q_{n-1}$ are even Hamiltonian and 
$TQ_n$, $CQ_n$ and $MQ_n$ are Hamiltonian-connected, 
$K_{n}^{TQ_k}$, $K_{n}^{CQ_k}$, $K_{n}^{MQ_k}$, $Q_{n-1}^{TQ_k}$, $Q_{n-1}^{CQ_k}$, 
$Q_{n-1}^{MQ_k}$ are all Hamiltonian by Theorem \ref{Ham-Th2}. 
Since these exponential graphs have nice properties similarly to the DCell network, 
they may be also considered as nice candidates for
doubly exponential-scale communication networks.

\begin{table}[b]
\caption{Comparison of $D_{k,n}$, $K_{n}^{TQ_k}$, $K_{n}^{CQ_k}$, $K_{n}^{MQ_k}$,
$Q_{n-1}^{TQ_k}$, $Q_{n-1}^{CQ_k}$ and $Q_{n-1}^{MQ_k}$ for even $n$.}
\begin{tabular*}{\textwidth}{@{\extracolsep\fill}lccc}
\toprule
& $D_{k,n}$ & $K_{n}^{TQ_k}$,\ $K_{n}^{CQ_k}$,\ $K_{n}^{MQ_k}$ 
& $Q_{n-1}^{TQ_k}$, $Q_{n-1}^{CQ_k}$, $Q_{n-1}^{MQ_k}$
\\
\midrule
Order & $\leq (n+1)^{2^{k}}-1$ & $2^{k}n^{2^{k}}$ & $2^{2^{k}(n-1)+k}$ \\ 
Degree & $n+k-1$ & $n+k-1$ & $n+k-1$\\ 
Diameter & $\leq 2^{k+1}-1$ & $2^{k+1}$ & $n2^{k}$\\ 
Connectivity & $n+k-1$ & $n+k-1$ & $n+k-1$\\
Hamiltonicity & Yes & Yes  & Yes \\ 
\bottomrule
\end{tabular*}
\end{table}

\begin{table}[t]
\begin{threeparttable}
\caption{Comparison of $D_{k,n}$, $K_{n}^{TQ_k}$, $K_{n}^{CQ_k}$, $K_{n}^{MQ_k}$,
$Q_{n-1}^{TQ_k}$, $Q_{n-1}^{CQ_k}$ and $Q_{n-1}^{MQ_k}$ for $n \leq 4$ and $k \leq 4$.}
\begin{tabular*}{\textwidth}{@{\extracolsep\fill}llrrrrrrr}
\toprule
& & 
& \multicolumn{2}{@{}c@{}}{$D_{k,n}$} 
& \multicolumn{2}{@{}c@{}}{$K_{n}^{TQ_k}$,\ $K_{n}^{CQ_k}$,\ $K_{n}^{MQ_k}$} 
& \multicolumn{2}{@{}c@{}}{$Q_{n-1}^{TQ_k}$,\ $Q_{n-1}^{CQ_k}$,\ $Q_{n-1}^{MQ_k}$}
\\
\cmidrule{4-5}
\cmidrule{6-7}
\cmidrule{8-9}
$n$  & $k$ & Deg.
& Order & Diam.\tnote{1}
& Order & Diam. 
& Order & Diam. \\
\midrule
2 & 1 & 2 & 6 & 3 & 8 & 4 & 8 & 4 \\ 
3 & 1 & 3 & 12 & 3 & 18 & 4 & 32 & 6 \\ 
4 & 1 & 4 & 20 & 3 & 32 & 4 & 128 &  8  \\ 
2 & 2 & 3 & 42 & 7 & 64 & 8 & 64 & 8 \\ 
3 & 2 & 4 & 156 & 7 & 324 & 8 & 1,024 & 12 \\ 
4 & 2 & 5 & 420 & 7 & 1024 & 8  & 16,384 & 16  \\ 
2 & 3 & 4 & 1,806 & 15 & 2,048 & 16  & 2,048 & 16  \\ 
3 & 3 & 5 & 24,492 & 15 & 52,488 & 16 & 524,288 & 24 \\ 
4 & 3 & 6 & 176,820 & 15 & 524,288 & 16 & 134,217,728 & 32 \\ 
2 & 4 & 5 & 3,263,442 & 31 & 1,048,576 & 32 & 1,048,576 & 32 \\ 
3 & 4 & 6 & 599,882,556 & 31 & 688,747,536 & 32 & 68,719,476,736 & 48 \\ 
4 & 4 & 7 & 31,265,489,220 & 31 & 68,719,476,736 & 32 & 4,503,599,627,370,496 & 64 \\ 
\bottomrule
\end{tabular*}
\begin{tablenotes}
\item[1] Upper bounds on the diameter
\end{tablenotes}
\end{threeparttable}
\end{table}

Table 5 compares $D_{k,n}$, $K_{n}^{TQ_k}$, $K_{n}^{CQ_k}$, $K_{n}^{MQ_k}$,
$Q_{n-1}^{TQ_k}$, $Q_{n-1}^{CQ_k}$ and $Q_{n-1}^{MQ_k}$ for $n \leq 4$ and $k \leq 4$.
Except for the case that $n = 2$ and $k = 4$,
all of the exponential graphs have more vertices than the DCell network.
We here remark that the diameters for the DCell network in Table 7
are upper bounds, some of which are shown to be equal to the diameter in \cite{Kli-et-al}, 
however.
Moreover, in \cite{Kli-et-al}, generalized DCell networks which have smaller diameter than
the original DCell network for some small $n$ and $k$ 
have been proposed, although their exact diameters are unknown.

\subsection{Iterated Exponential Graphs}

\begin{table}[b]
\centering
\caption{Fundamental properties of $\Omega(G,k)$ and $\Psi(G,k)$, where $G$ is a connected
graph of order $n \geq 2$ and $k \geq 2$.}
\begin{tabular*}{\textwidth}{@{\extracolsep\fill}lcc}
\toprule
& $\Omega(G,k)$ & $\Psi(G,k)$ \\ 
\midrule
Order & $n^{\frac{n^{k}-1}{n-1}}$ & 
$n^{n^{n^{\iddots^{n^{n+1}}}+ \cdots+n^{n^{n+1}+n+1}+ n^{n+1}+n+1}
+n^{\iddots^{n^{n+1}}}+\cdots +n^{n^{n+1}+n+1}+ n^{n+1}+n+1}$ \\ 
Minimum degree & $k \cdot  \delta(G)$ & $k \cdot \delta(G)$ \\ 
Maximum degree & $k \cdot \Delta(G)$ & $k \cdot \Delta(G)$ \\ 
Diameter & $n^{k-1}{\rm diam}(G)+\left(\frac{n^{k-1}-1}{n-1}\right) {\rm diam}^\ast(G)$ 
& $\leq ({\rm diam}(G)+2) \cdot |V(\Psi(G,k-1))|-2$ \\ 
Connectivity & $k \cdot  \delta(G)$ & $k \cdot \delta(G)$ \\ 
Hamiltonicity & $G$: even Hamiltonian-connected & $-$ \\
EDHCs \& CISTs & $G \cong K_n$ \ (even $n \geq 4$) & $-$ \\
\bottomrule
\end{tabular*}
\end{table}

For any given graph $G$, we can iteratively apply the exponential operation. 
Define iterated exponential graphs $\Omega(G,k)$ and $\Psi(G,k)$ as follows:
\[ \left\{ \begin{array}{lcll}
\Omega(G,1) & = & G, & \\[1mm]
\Omega(G,k) & = & \Omega(G,k-1)^{G} & \mbox{ for } k \geq 2.
\end{array} \right.
\]
\[ \left\{ \begin{array}{lcll}
\Psi(G,1) & = & G, & \\[1mm]
\Psi(G,k) & = & G^{\Psi(G,k-1)} & \mbox{ for } k \geq 2.
\end{array} \right.
\]

Fundamental properties of these iterated exponential graphs for a nontrivial 
connected graph $G$ of order $n$ are summarized in Table 6.
Since 
$|V(\Psi(G,k)| = O\left(\underbrace{n^{n^{\iddots^{n}}}}_{k}\right)$,
the order of $\Psi(G,k)$ is enormous and it seems impractical 
unless $n$ and $k$ are small.
In particular, we denote by $\Omega_{n,k}$ and $\Psi_{n,k}$ the iterated exponential graphs
$\Omega(K_n,k)$ and $\Psi(K_n,k)$, respectively. 
For $\Omega_{n,k}$, we can apply Theorems \ref{Ham-Th2} and \ref{Ham-Th3}.
Thus, $\Omega_{n,k}$ has two edge-disjoint Hamiltonian cycles (EDHCs) and
two completely independent spanning trees (CISTs) when $n$ is even and $n \geq 4$.
For the DCell network, it has been shown in \cite{Qin-et-al}
that $D_{k,n}$ has two completely independent spanning trees when $n \geq 6$.
We also simply denote by $\Omega_{k}$ and $\Psi_{k}$ for 
$\Omega_{2,k}$ and $\Psi_{2,k}$, respectively,
and call $\Omega_{k}$ and $\Psi_{k}$ the {\it exponential cube}
and {\it hyper-exponential cube}, respectively.  
It can be easily checked that $\Omega_2 \cong C_8$.
Thus, the exponential graph shown in Fig. 1 is just the exponential cube $\Omega_3$.

\begin{table}[t]
\centering
\caption{Comparison of $D_{k,n}$, $\Omega_{k}$ and $\Psi_{k}$ where $k \geq 2$.}
\begin{tabular*}{\textwidth}{@{\extracolsep\fill}lccc}
\toprule
& $D_{k,n}$ & $\Omega_{k}$ & $\Psi_{k}$ \\
\midrule
Order & $\leq (n+1)^{2^{k}-1}$ & $2^{2^{k}-1}$ & 
$2^{2^{2^{\iddots^{2^{3}}}+ \cdots+2^{11}+8+3}
+2^{\iddots^{2^{3}}}+\cdots +2^{11}+8+3}$  \\ 
Degree (regular) & $n+k-1$ & $k$ & $k$ \\
Diameter & $\leq 2^{k+1}-1$ & $3 \cdot 2^{k-1}-2$ 
& $\leq 3 \cdot 2^{2^{\iddots^{2^{3}}}+ \cdots+2^{11}+8+3} -2$ \\
Connectivity & $n+k-1$ & $k$ & $k$ \\
Hamiltonicity & Yes & Yes & $-$ \\
\bottomrule
\end{tabular*}
\end{table}

\begin{table}[t]
\caption{Comparison of $D_{k,2}$, $\Omega_{k+1}$ and $\Psi_{k+1}$ for $k \leq 5$.}
\hspace*{-20mm}
\begin{threeparttable}
\begin{tabular}{lrrrrrrr}
\toprule & 
& \multicolumn{2}{@{}c@{}}{$D_{k,2}$} 
& \multicolumn{2}{@{}c@{}}{$\Omega_{k+1}$}
& \multicolumn{2}{@{}c@{}}{$\Psi_{k+1}$} \\
\cmidrule{3-4}
\cmidrule{5-6}
\cmidrule{7-8}
$k$  & Deg. 
& Order & Diam.\tnote{1}
& Order & Diam. 
& Order & Diam.\tnote{1} \\
\midrule
1 & 2 & 6 & 3 & 8 & 4 &  8 &  4 \\ 
2 & 3 & 42 & 7 & 128 & 10 & 2,048  & 22 \\ 
3 & 4 & 1,806 & 15 & 32,768 & 22 & $2^{2059}$  & 6142 \\ 
4 & 5 & 3,263,442 & 31 & 2,147,483,648  & 46 & $2^{2^{2059}+2059}$ 
& $3 \cdot 2^{2059}-2$ \\ 
5 & 6 & 10,650,056,950,806 & 63 & 9,223,372,036,854,775,808 & 94
& $2^{2^{2^{2059}+2059}+2^{2059}+2059}$ & $3 \cdot 2^{2^{2059}+2059}-2$ \\ 
\bottomrule
\end{tabular}
\begin{tablenotes}
\item[1] Upper bounds on the diameter
\end{tablenotes}
\end{threeparttable}
\end{table}

Comparison of $D_{k,n}$, $\Omega_{k}$ and $\Psi_{k}$ is given in Table 7. 
Compared to $D_{k,n}$, $\Omega_k$ has advantages that both the order and the diameter
are explicitly determined and in addition, the definition of $\Omega_k$ is simpler than
that of $D_{k,n}$.
Thus, the exponential cube $\Omega_k$ may be considered as another nice candidate 
for doubly exponential-scale communication networks.

Table 8 compares $D_{k,2}$, $\Omega_{k+1}$ and $\Psi_{k+1}$ for $k \leq 5$. 
Note that all these networks have the same degree.
The exponential cube has more vertices than the DCell network, while
the DCell network has smaller diameter than the exponential cube.
It may be useful to apply these networks (and also $D_{k,n}$ and $\Omega_{n,k}$) 
complementary depending on the required number of vertices of a network.
On the other hand, we can see that the hyper-exponential cube $\Psi_{k+1}$ has actually 
enormous order even if $k = 4$.
As far as we know, no such huge network with logarithmic diameter has been 
defined in graph theory until now.
From a theoretical rather than a practical point of view,
the hyper-exponential cube $\Psi_{k}$ would be of interest.

\section{Concluding Remarks}

In this paper, we have newly introduced the notion of exponentiation of graphs
and studied structural properties of exponential graphs 
such as diameter, connectivity and Hamiltonicity.
We have also applied our results to constructions of 
maximally connected and super edge-connected Hamiltonian networks
of multi-exponential order with logarithmic diameter
and furthermore compared them to the DCell network which is well-known as a communication
network of doubly exponential order with logarithmic diameter.

Since the DCell network has doubly exponential order,
there are large gaps between the orders of those with different settings
of parameters.
Using the exponential operation, we can construct multi-exponential scale networks
with various orders and thus obtain a flexibility for 
constructing very large-scale networks with logarithmic diameter; 
in fact, we can also apply the exponential operation to the DCell network.

While the exponential operation on graphs is a useful tool for constructing
very large-scale networks, 
the operation itself may be of theoretical interest. 
It would be interesting to study other properties (e.g., symmetry, Hamiltonian-connectedness,
fault-tolerance, diagnosability) and applications of exponential graphs 
from the theoretical or practical point of view.

\section*{Acknowledgements}

This work was supported by JSPS KAKENHI Grant Number JP19K11829.

\end{document}